\documentclass[11pt,runningheads,a4paper]{llncs}

\usepackage{amssymb}
\usepackage{graphicx}
\usepackage{amsmath}
\usepackage[colorinlistoftodos]{todonotes}

\presetkeys%
{todonotes}%
{backgroundcolor=yellow}{}
\usepackage{bbm}
\usepackage{centernot}
\usepackage{enumerate}
\usepackage[a4paper, margin=1.5in]{geometry}
\usepackage{cite}
\usepackage{setspace}

\usepackage{url}
\urldef{\mailsa}\path|{blake.mcgranecorrigan.2017@mumail.ie}|

\begin{document}

\mainmatter  

\title{A density-dependent metapopulation model: Extinction, persistence and source-sink dynamics}
\titlerunning{A density-dependent metapopulation model}

\author{\normalsize Blake McGrane-Corrigan\thanks{\textit{Corresponding Author (blake.mcgranecorrigan.2017@mumail.ie) \hfill}}, Oliver Mason, Rafael de Andrade Moral}

\authorrunning{A density-dependent metapopulation model}

\institute{\normalsize Department of Mathematics and Statistics,\\
Maynooth University, Kildare, Ireland}

\toctitle{A density-dependent metapopulation model}
\tocauthor{}

\maketitle

\begin{abstract}
We consider a nonlinear coupled discrete-time model of population dynamics. This model describes the movement of populations within a heterogeneous landscape, where the growth of subpopulations are modelled by (possibly different) bounded Kolmogorov maps and coupling terms are defined by nonlinear functions taking values in $(0,1)$. These couplings describe the proportions of individuals dispersing between regions. We first give a brief survey of similar discrete-time dispersal models. We then derive sufficient conditions for the stability/instability of the extinction equilibrium, for the existence of a positive fixed point and for ensuring uniform strong persistence. Finally we numerically explore a planar version of our model in a source-sink context, to show some of the qualitative behaviour that the model we consider can capture: for example, periodic behaviour and dynamics reminiscent of chaos.

\vspace{3mm}


\end{abstract}

\section{Introduction}
Habitat fragmentation caused by processes such as deforestation, urbanisation and agriculture, has caused a diverse range of species to become so-called \textit{metapopulations}, sub-groups within a population spread out across many regions in a patchy landscape \cite{cayuela, osloy, scriven}. A key ecological question is then how dispersal between these regions impacts such species' population dynamics and if this can promote coexistence among regions. For example, given that a subpopulation on one or more of the isolated regions goes extinct, one may be interested in knowing if the coupled system is persistent, if such a system possesses a fixed point or if there exists a locally/globally stable positive fixed point. In population dynamics coupling regional dynamics via dispersal has attracted a lot of attention in recent literature, albeit primarily in continuous time. Recently there has been more interest in discrete-time dispersal models \cite{bajo, kirkland, ruiz1, ruiz2, ecomod}. In many models the proportion dispersing is assumed to be constant. However it has been found that this can also largely depend on the density of the region from where such movement is taking place \cite{harman, travis}.

\vspace{3mm}

In this paper we will consider a model closely related to that of \cite{yakubo} and \cite{ruiz1}, which includes nonlinear dispersal rates, where each regional map can come from a different class of population maps and where there is a cost to dispersal. This builds on the work of \cite{ecomod}, where the authors considered a planar Ricker model with density-dependent (cost-free) dispersal in the context of invasion dynamics. Here, we will consider an arbitrary, finite number of regions, where regional dynamics can be non-identical and come from a more general class of population maps than in \cite{ecomod}. We will assume that regional dynamics are modelled by bounded population maps and determine the conditions that ensure stability and persistence after regions are connected by dispersal. We will first give some preliminary results and background. We will then introduce our model and state our main theoretical and numerical results.

\section{Preliminaries}
In this section we fix some notation and terminology, and recall key definitions and results used throughout the paper. 

\vspace{3mm}

For $x \in \mathbb{R}^n$ we write $x \geq 0$ $(\gg 0)$ to mean that $x_i \geq 0$ $(>0)$ \mbox{$\forall$} $ i \in \{1,...,n\}$. Denote by $\mathbb{R}_+^n := \{x \in \mathbb{R}^n : x \geq 0 \}$ the standard \textit{non-negative cone} in $\mathbb{R}^n$. A similar notation is used for matrices in $\mathbb{R}^{n \times n}$ and $\mathbb{R}^{n \times n}_+$. For $A,B \in \mathbb{R}^{n \times n}$ we write $A \geq B$ if $A-B \in \mathbb{R}^{n \times n}_+$. Denote by $A^T$ the transpose of $A$ and $\sigma(A)$ the spectrum of $A$ (the set of its eigenvalues). Denote by $F^{m} = F \circ F \circ \ \cdots\  \circ F$ ($m$ times) the \(m\)-fold composition of a map \(F:\mathbb{R}^{n} \rightarrow \mathbb{R}^{n}\) with itself, \(m \in \mathbb{Z}_{+}\). The \textit{Jacobian} of a vector-valued function $f:\mathbb{R}^n \rightarrow \mathbb{R}^n$ at $a \in \mathbb{R}^n$ is the matrix \[f'(a) := \left(\frac{\partial f_i}{\partial x_j}(a)\right), \ i,j \in \{1,...,n\},\] provided each partial derivative of $f$ exists. A matrix $A=(a_{ij}) \in \mathbb{R}^{n \times n}$ is said to be \textit{irreducible} if there is no proper subset $J \subset \{1,...,n\}$ for which $a_{ij} = 0$ \mbox{$\forall$ }$i \in J,j \in \{1,...,n\} \backslash J$. 
$A \in \mathbb{R}^{n \times n}$ is said to be \textit{column substochastic} if $\mathbbm{1}^TA < \mathbbm{1}^T$, where $\mathbbm{1} := (1,...,1)^T$. $A \in \mathbb{R}^{n \times n}$ is said to be \textit{primitive} if $\exists \ s \in \mathbb{N}$ such that $A^s \gg 0.$ The \textit{spectral radius} and the \textit{spectral abscissa} of $A \in \mathbb{R}^{n \times n}$ are respectively defined as \begin{eqnarray*}
    \rho(A) &:=& \max\{ |\lambda| : \lambda \in \sigma(A) \} \\
    \mu(A) &:=& \max\{ \mbox{Re}(\lambda) : \lambda \in \sigma(A) \}.
\end{eqnarray*} For more background on matrix theory see \cite{hj}. 

\vspace{3mm}

Given the initial value problem \begin{equation} \label{IVP}
\begin{aligned}
    &x(t+1) = f(x(t)), \\ 
    &x(0)=x_0 \in \mathbb{R}^n, 
\end{aligned} 
\end{equation} for $x \in \mathbb{R}^n$, $f:\mathbb{R}^n \rightarrow \mathbb{R}^n$ a smooth function, let \(x(t,x_{0})\) denote the solution of (\ref{IVP}) at time \(t\) corresponding to
\(x(0) = x_{0}\). An \textit{equilibrium} of such an initial value problem is a solution of \(f(x)=x\). An equilibrium \(x^{\star}\) is (\textit{Lyapunov}) \textit{stable} if for any \(\epsilon > 0\) $\mbox{there exists}$ a \(\delta > 0\) such that \[\left\| x_{0} - x^{\star} \right\| < \delta \implies \left\| x( t, x_0 ) - x^{\star} \right\| < \epsilon \mbox{ $\forall$ } t\geq0.\] If, in addition, there exists some $R>0$, such that $x(t, x_0)\rightarrow x^{*}$ as $t \rightarrow \infty$ for any solution with $\|x_0\|<R$, the equilibrium is \textit{locally asymptotically stable} (\textit{LAS}). If this holds for any $R>0$, it is said to be \textit{globally asymptotically stable} (\textit{GAS}). It is common to assume that $f(0)=0$ and thus $0$ is an equilibrium of (\ref{IVP}). In the context of population dynamics, this trivial equilibrium corresponds to extinction.

\vspace{3mm}

Assume that (\ref{IVP}) is a \textit{positive system}, i.e. $x_0 \in \mathbb{R}_+^n \implies x(t,x_{0}) \in \mathbb{R}_+^n$. Given a \textit{persistence function} \(\eta:\mathbb{R}_+^{n} \rightarrow \left\lbrack 0,\infty \right)\), the system (\ref{IVP}) is \textit{uniformly weakly} $\eta$-\textit{persistent}, if $\mbox{there exists}$ some \(\epsilon > 0\), such that \begin{eqnarray} \label{limsup}
    \eta(x_{0} ) > 0 \implies \limsup_{t \rightarrow \infty} \eta \left( x( t, x_0 ) \right) > \epsilon.
\end{eqnarray} If the $\limsup$ is replaced with $\liminf$ in (\ref{limsup}) it is referred to as \textit{uniform} \textit{strong} \(\eta\)-\textit{persistence}. For more background on persistence theory see \cite{st}.

\vspace{3mm}

The next result is the well-known Perron-Frobenius theorem.
 
\begin{theorem} \label{pf} \cite{hj} Let $A \in \mathbb{R}_{+}^{n \times n}$ be irreducible. Then, $\rho(A)>0$ is an eigenvalue of $A$, and $\mbox{there exists}$ a unique (up to scalar multiple) vector $v \gg 0$ $(w^T \gg 0)$ such that $Av = \rho(A)v$ $(w^TA = \rho(A)w^T)$.
\end{theorem}

A simple consequence of Theorem \ref{pf} is recalled in the following lemma.
 
\begin{lemma} \label{lem1} Let $A \in \mathbb{R}_{+}^{n \times n}$. Then
    \[\rho(A)< 1 \ (\rho(A)>1) \iff \exists \ v \gg 0 \mbox{ such that } Av \ll v \ (Av\gg v).\]

\end{lemma}

We will now gather two results on persistence from \cite{st}. Let \begin{eqnarray} \label{genmatrixmodel} x(t+1) = H(x(t)) := A(x(t))x(t), \ \ x(0)=x_0 \in \mathbb{R}_+^{n \times n}\end{eqnarray} be a general nonlinear matrix model, where $A:\mathbb{R}_+^n \rightarrow \mathbb{R}_+^{n \times n}$.


\begin{theorem} \label{stthm} \cite{st} Let
\(H:\mathbb{R}_{+}^{n} \rightarrow \mathbb{R}_{+}^{n}\) be
differentiable. Let \(\eta (x) = \left\| x \right\|\), where
\(\left\| \cdot \right\|\) is any norm on \(\mathbb{R}_{+}^{n}\). Suppose that the following hold:
\begin{enumerate}
\def\labelenumi{\arabic{enumi}.}
\item
  \(H\left(\mathbb{R}_{+}^{n}\backslash\text{\{}0\}\right) \subset \mathbb{R}_{+}^{n}\backslash\text{\{}0\} \);
\item $\mbox{there exists}$ \(r_{0} > 1\) and \(v \gg 0\) such that $H'\left( 0 \right)^{T}v \geq r_{0}v$;
  \item $\mbox{there exists}$ \(M > 0\) such that \mbox{$\forall$} \(x(0) \in \mathbb{R}_+^{n}\) $\mbox{there exists}$ \(T \in \mathbb{N}\) such that \(\|x(t,x_{0})\| \leq M\) \mbox{$\forall$} \(t \geq T\).
\end{enumerate} Then, $H$ is uniformly \(\eta\)-persistent. Let \(\hat{\eta}\left(x \right) = \text{min}_{i}\ x_{i}\) and define \begin{eqnarray*} X_{0} := \{x_{0} \in \mathbb{R}_{+}^{n} :\hat{\eta}\left(x(t,x_{0}) \right) = 0, \ \forall \ t \geq 0.\}.\end{eqnarray*} In addition, if for all \(c > 0\) $\mbox{there exists}$ some \(s > 0\) such that $H^s\left( x \right) \gg 0 \ \forall\ x \in \mathbb{R}_+^{n}, \ 0 < \|x(s)\| \leq c$, then $\mbox{there exists}$ some \(\epsilon > 0\) such that \[\liminf_{t \rightarrow \infty} \hat{\eta} \left( x( t, x_0 ) \right) \geq \epsilon\] for any \(x_0 \in \mathbb{R}_{+}^{n}\backslash X_{0}\).
\end{theorem}

\section{Dispersal in Discrete-Time}
Before we introduce our model, we will first briefly describe several discrete-time dispersal models that have been proposed in the literature. This is to put our work in context, and to clarify how the model class we study here relates to others considered in the mathematical ecology literature. 

\vspace{3mm}

In \cite{yakubo} the authors investigated the $n$-dimensional system
\begin{equation} \label{yakubodisp}
\begin{aligned} &x_i(t+1) = \left(1 - \sum_{j =1}^n d_{ij}\right)f_i(x_i(t)) + \sum_{j =1}^n d_{ji}f_j(x_j(t)), \\ \ \ &x_i(0) \in \mathbb{R}_+,
\end{aligned}
\end{equation} where for each $i \neq j$, $0<d_{ij}, d_{ji}< 1$, $\sum_{j=1}^nd_{ij} \in (0,1)$ and $d_{ii}=0$. Each $d_{ij}$ is the proportion of individuals dispersing from region $i$ to $j$. They also assumed that each $f_i$ was given by a so-called $\alpha$-\textit{monotone} \textit{concave map}. Let $\alpha \in (0, \infty]$. A positive $C^2$ map, $f$, is $\alpha$-concave monotone if \[f'(x)>0, f''(x)<0 \ \forall \ x \in [0,\alpha].\] Examples of such maps can be generated by choosing appropriate parameter values for the \textit{Ricker}, \textit{Smith-Slatkin} and \textit{Beverton-Holt} maps. The authors then went to give a sufficient condition for the existence and global stability of a positive fixed point of (\ref{yakubodisp}).

\vspace{3mm}

In \cite{kirkland} the authors proposed the following coupled model 
\begin{equation} \label{kirkmodel}
\begin{aligned}
    &x(t+1) = S_p \Gamma(x(t))x(t), \\ &x(0) \in \mathbb{R}_+^n 
\end{aligned} 
\end{equation} where $p \in (0,1]^n$, $S=(s_{ij}) \in \mathbb{R}_+^{n \times n}$ is a primitive column substochastic matrix, $\Gamma:\mathbb{R}_+^n \rightarrow \mathbb{R}_+^{n \times n}$ is given by $\Gamma(x)=\mbox{diag}(g_1(x_1),...,g_n(x_n))$, and \[S_p := I- \mbox{diag}(p) + S\mbox{diag}(p).\] Each $g_i: \mathbb{R}_+ \rightarrow \mathbb{R}_+$ is a positive, continuous, decreasing map such that \[\lim_{x_i \rightarrow \infty}g_i(x_i)<1\] and $f_i(x)=g_i(x)x$ is increasing. In relation to (\ref{kirkmodel}), the authors stated the following result, the proof of which relied heavily on the strong monotonicity properties of their system class.

\begin{theorem}\cite{kirkland} \label{kirkthm}
    If $\rho(S_p \Gamma(0)) \leq 1$ then the extinction equilibrium for (\ref{kirkmodel}) is GAS. If $\rho(S_p \Gamma(0)) > 1$ then there exists a GAS positive equilibrium for (\ref{kirkmodel}).
\end{theorem}

A variation of the models of \cite{kirkland} and \cite{yakubo} was studied in \cite{ruiz1} and \cite{ruiz2}. In these papers, the model studied is of the general form \begin{equation} \label{rs}
\begin{aligned} x_i(t+1) &= \sum_{j =1}^n d_{ij}f_j(x_j(t)),
\end{aligned}
\end{equation} where $x_i(0) \in \mathbb{R}_+^n$. In \cite{ruiz1} the author assumed that $f_1=\cdots f_n = f$, where $f(x)=g(x)x$ for $g: \mathbb{R}_+ \rightarrow \mathbb{R}_+$. They also assumed that $\sum_j d_{ij}=1$. Under the additional assumption that $d_{ij}=d_{ji}$ for all $i,j \in \{1,...,n\}$, they proved that if $x^* \geq 0$ is GAS for $f$ then $(x^*, ..., x^*) \in \mathbb{R}_+^n$ is GAS for (\ref{rs}). In \cite{ruiz2} the author looked at (\ref{rs}) where each region depended on some intrinsic parameter $r_i \geq 0$. The author then looked numerically at various scenarios involving different dispersal mechanisms when restricted to two regions. An early instance of a model of the form in (\ref{rs}) can be found in \cite{karlin}, where the local stability properties of a planar system in the applied context of mathematical genetics were studied.

\vspace{3mm}

In \cite{franco},  \cite{grombach} and \cite{vortkamp} the authors studied similar models to \cite{kirkland} and \cite{yakubo} for when $n=2$, i.e.
\begin{equation} \label{vortkampmodel}
\begin{aligned}
    &x_i(t+1) =  \left(1-d_i\right)f_i(x_i(t)) +  d_jf_j(x_j(t)), \\ 
    &x_i(0) \in \mathbb{R}_+, i \in \{1,2\}, i \neq j.
\end{aligned}
\end{equation}
In \cite{franco} they assumed that $f_i(x) = r_i x g_i(x)$, where $g:\mathbb{R}_+ \rightarrow (0,\infty)$ was strictly decreasing and $g_i(0)=1$. They showed that, when the spectral radius is less than 1, any initial population is driven to extinction. On the other hand, they showed that (\ref{vortkampmodel}) is permanent (there is a compact set $K \subset \mbox{Int}(\mathbb{R}_+^2)$ and $t_0>0$ such that any solution $x(t, x_0)$ remains in $K$ for all $t \geq t_0$) if the spectral radius of the Jacobian of the system at the extinction equilibrium is greater than $1$.  The model was investigated numerically for various parameter scenarios when regional dynamics are given by either a Ricker or Hassell-1 map, i.e. \[f_i(x)=\dfrac{a_ix_i}{1+b_ix_i}.\] In \cite{vortkamp} they assumed that $d_1=d_2=d \in [0,0.5]$ and $f_1=f_2=f$, where \[f(x) = x \exp\left(r\left(1-\frac{x}{K}\right)\left(\frac{x}{A}-1\right)\right),\] $r>0$ and $0<A<K$. Here $A$ and $K$ are respectively the so-called \textit{Allee threshold} and \textit{carrying capacity}. They found that the appearance and disappearance of attractors of their coupled system is dependent on how strong the level of dispersal is. They also explored how transient phenomena emerged within such a simple coupled system. More recently, in \cite{grombach} the authors assumed that the dispersal rate was symmetric, i.e. $d_1=d_2=d \in [0,1]$ and that each $f_i$ was given by a Hassell-1 map, with each $a_i>1$ (region $i$ is a so-called source) and $b_i > 0$. They investigated how changing the dispersal rate in various scenarios affected the asymptotic total population size and discussed the biological interpretations of their results.

\vspace{3mm}

In \cite{ecomod} the authors proposed a similar planar model to \cite{grombach} and \cite{vortkamp}, in order to model insect dynamics between two hosts, where dispersal was asymmetric and given by an arbitrary nonlinear function taking values in $(0,1)$. This took the general form \begin{equation} \label{ecommodel}
\begin{aligned}
    &x_i(t+1) =  \left(1-d_i(x_i(t))\right)f_i(x_i(t)) +  d_j(x_j(t))f_j(x_j(t)), \\ 
    &x_i(0) \in \mathbb{R}_+
\end{aligned}
\end{equation} for $i,j \in \{1,2\}, i \neq j$. Regional growth is modelled by \[f_i(x) := R_iF_iS_i x \mbox{exp}(-\mu x),\] where $R_i \in (0,1), S_i \in (0,1), F_i>0$ and  $\mu > 0$ are respectively the \textit{sex ratio}, \textit{survival}, \textit{fecundity} and \textit{intraspecific competition} parameters. Dispersal was given by a smooth map $d_k:\mathbb{R}_+^2 \rightarrow (0,1)$, for $k \in \{1,2\}$ with $d_i \neq d_j$ for $i \neq j$. The authors gave sufficient conditions for persistence and existence of a positive fixed point, which were robust to choices of $d_{k}$. They also found that for specific parametrisations, even though one region may go extinct in isolation, if the two regions are connected by heterogeneous dispersal one can induce a so-called rescue effect, preventing the declining region from going extinct.

\section{The Metapopulation Model}

We now define the model class to be studied here and highlight how it relates to some of the models discussed above. 

\vspace{3mm}

Consider a population that inhabits $n \in \mathbb{Z}_+ := \mathbb{N} \cup \{0\}$ regions within some landscape and denote the population density in region $i\in\{1,...,n\}$, at time $t \in \mathbb{Z}_+$, by $x_i(t) \in \mathbb{R}^n_+$. We implicitly assume that all regions are accesible by all individuals. Let \(f_{i}:\mathbb{R}_+ \rightarrow \mathbb{R}_+\) be a map which respectively describes population growth in region \(i\). Let $\cal{M}$ denote the set of maps $f:\mathbb{R}_+ \rightarrow \mathbb{R}_+$ such that the following hold:
\begin{enumerate}[(A)]
    \item $f$ is $C^1$ and \textit{positive definite}, i.e. $f(0)=0$ and $f(x)>0$ \mbox{for} $x > 0$;
    \item $f$ is a \textit{Kolmogorov-type map}, i.e. $f(x) := g(x)x$ and $g: \mathbb{R}_+ \rightarrow (0, \infty)$ is $C^1$ \cite{kon};
    \item $\exists$ $m \in [0, \infty)$ : $f(x) \leq m$ $\forall$ $x \geq 0$.
\end{enumerate} 

Throughout the rest of the paper we will assume that \[\{f_1,...,f_n\} \subset \mbox{$\cal{M}$}.\]
The class $\cal{M}$ contains many of the commonly used maps for discrete time modeling of ecological systems. In particular, the following all belong to $\cal{M}$ \cite{francoh, hassell, ricker, schreiber01, schreiber03}:
\begin{eqnarray*}
   &&\bullet \ \mbox{Generalised Beverton-Holt: } f(x)=\dfrac{ax}{1+(x/b)^{c}}; \ a>0, b> 0, c \geq 1. \\
    &&\bullet \ \mbox{Hassell: } f(x)=\dfrac{ax}{(1+bx)^{c}}; \ a,b>0, c \geq 1. \\
    &&\bullet \ \mbox{Ricker: } f(x)= ax\exp(-bx); \ a,b>0. \\
    &&\bullet \ \mbox{Logistic: } f(x)= ax(1-x); a \in [0,4].\end{eqnarray*}


\begin{remark}Let $f:\mathbb{R}_+ \rightarrow \mathbb{R}_+$ be a $C^1$ map that satisfies:
\begin{enumerate}[(a)]
\item $f$ is positive definite.
\item $f$ has fixed points $\{0,K\}$, such that $K \in (0, \infty)$, $f(x)>x$ for $x \in (0,K)$ and $f(x)<x$ for $x \in (K,\infty)$.
\item $f$ has a unique critical point $L < K$ such that $f'(x) >0$ \mbox{$\forall$} $x \in (0,L)$, $f'(x)<0$ \mbox{$\forall$} $x \in (L, \infty)$, and $f'(0) >0.$
\end{enumerate}

Such maps are known as \textit{unimodal population maps} \cite{francoh}. Let $\cal{U}$ be the set of unimodal population maps. It is clear, by definition, that the following inclusion holds: \[\mbox{$\cal{U}$} \subset \mbox{$\cal{M}$}.\] The stability of systems with dynamics given by maps in $\cal{U}$ has attracted much interest in the mathematical ecology literature. For example, see  \cite{sarkovskii, cull88, schreiber01, schreiber03}.
\end{remark}

Following reproduction/recruitment, sub-populations from region $j$ move to region $i$, with the proportion given by the function \(d_{ij}:\mathbb{R}_+ \rightarrow (0,1)\), where \begin{enumerate}[(D)] \item $d_{ij}$ is $C^1$ and for each fixed $i \in \{1,...,n\}$, $\sum_{j=1}^{n} d_{ji}(x) < 1$ \mbox{$\forall$} $x \geq 0$. \end{enumerate} 
A general form for such a coupled system is
\begin{equation} \label{gendisp}
\begin{aligned} x_i(t+1) &= \sum_{j =1}^n d_{ij}(x_j(t))f_j(x_j(t)),
\end{aligned}
\end{equation} where $x_i(0) = \bar{x}_i \in \mathbb{R}_+^n$, $i \in \{1,...,n\}$. We can rewrite (\ref{gendisp}) as a so-called nonlinear matrix model\begin{gather} \label{matrixmodel} x(t+1) = F(x(t)) := A(x(t))x(t),\ \ x(0) = x_{0} \in \mathbb{R}_+^n,\end{gather} where the matrix valued function $A:\mathbb{R}_+^n \rightarrow \mathbb{R}_+^{n \times n}$ is defined as \begin{eqnarray} \label{matrix} A(x) := D(x)G(x), \end{eqnarray} where $D(x) := (d_{ij}(x_j))$ and $G(x) := \mbox{diag}(g_1(x_1), ..., g_n(x_n))$. As $g_i$ and $d_{ij}$ are continuous, $A$ is also continuous. Note that for a fixed $x \in \mathbb{R}_+^n$ we have that $D(x)$ is a column substochastic matrix.

\begin{remark} \label{remarkD}
    If $\sum_{j=1}^{n} d_{ji}(x) = 1$ \mbox{$\forall$} $x \geq 0$, then this could be interpreted as there being no cost to dispersal, which is a common assumption in models with constant dispersal \cite{bajo, ruiz1}. However, in this paper we assume (D) holds, so that \[d_{ii}(x) \neq 1-\sum_{j \neq i} d_{ji}(x),\] which allows one to incorporate dispersal costs, i.e. the sum of the proportions of individuals remaining and leaving patch $i$ is less than unity. Assumption (D), along with $\{f_1,...,f_n\} \subset \mbox{$\cal{M}$}$, further highlights the difference between our coupled model and the models of \cite{franco, vortkamp, grombach, yakubo, kirkland, ruiz1, ruiz2}. Note that the current authors' previous model in \cite{ecomod} does not allow for dispersal costs. Further note that each $d_{ij}$ being a function of $x$ means that one can incorporate density-dependent effects like competition, mortality and other ecological processes that may affect individuals when dispersing between regions. 
\end{remark}

In \cite{kirkland} and \cite{yakubo} the authors assumed dispersal was constant. In \cite{kirkland} they also had three assumptions on $g_i$ and $f_i$, as seen in Section 3. In \cite{yakubo} the authors assumed each region was modelled by an $\alpha$-concave monotone map. We note that our regional maps are different to those considered by \cite{yakubo}. We can demonstrate this with an example.

\begin{example}
    Consider the function \[f(x) = \gamma x \exp(-(x-1)^2),\] where $\gamma > e$. Clearly $f$ is continuous, $f(x)>0$ for $x>0$ and $f(0)=0$. Therefore (A) holds. $f$ is of the form $g(x)x$ with $g(x) = \gamma \exp(-(x-1)^2)$. Therefore, as $g$ is continuous, we have that (B) holds. $f$ also has a unique global maximum at $m = (1 + \sqrt{3}) / 2$ and so (C) holds. Therefore $f \in$ $\cal{M}$. One can verify that \begin{eqnarray*}
        f'(x) &=& \gamma \exp(-(x-1)^2)(1+2x-2x^2) \\ 
    f''(x) &=& 2\gamma \exp(-(x-1)^2) \left(2 -x -4x^{2} + 2x^{3}\right).
    \end{eqnarray*} From this we can see that $f'(x)>0 \ \forall \ x \in (0,m)$, but for $m_1 \approx 0.223$ we have that $f''(x)>0 \ \forall \ x \in (0,m_1)$. Therefore there does not exist an $\alpha \in (0, \infty]$ such that $f''(x)<0$ for all $x \in (0, \alpha]$ and so the definition of $\alpha$-concave monotonicity is violated. Hence we have found $f \in$ $\cal{M}$ such that $f$ is not $\alpha$-concave monotone.
\end{example}

\begin{remark} Note that maps in $\cal{U}$ are closely related to $\alpha$-concave monotone maps first introduced in \cite{yakubo}. Maps in $\cal{U}$ must possess a positive equilibrium that is greater than $L>0$. One must also have that they are increasing and concave up to this maximum, and decreasing and concave down after it. In \cite{yakubo} the authors assumed that the dynamics on each region was modelled by an $\alpha$-concave monotone map of the form $f(x):=g(x)x$, where $g$ is a strictly decreasing, positive, $C^2$ map. Clearly if $f$ is $\alpha$-concave monotone then (a) and (b) hold. For $f \in$ $\cal{U}$ assumption (c) ensures these maps are unimodal. In \cite{yakubo} they only assume $f'(x)>0$ up to $\alpha>0$. This does not require that $\alpha < K$ and so an $\alpha$-concave monotone map is not necessarily unimodal. If we took $\alpha=L$ as in the definition of $\cal{U}$ then the fact that $f$ is concave in $[0,\alpha]$ follows from (b) and (c).
\end{remark}

\section{Main Results}

We will now discuss some of the qualitative properties of (\ref{gendisp}).

\subsection{Stability}
 We first state a sufficient condition for LAS of the extinction equilibrium.

\begin{proposition} \label{max<1}
    Suppose $\{f_1,...,f_n\} \subset \mbox{$\cal{M}$}$ and $F(x)=A(x)x$, where $A(x)$ is given by (\ref{matrix}). Further assume that $\max_{i \in \{1,...,n\}}g_i(0)<1$. Then, $x^*=0$ is a LAS equilibrium of system (\ref{gendisp}) for all $d_{ij}$ satisfying (D).
\end{proposition}
\begin{proof}
Assume $\max_{i \in \{1,...,n\}}g_i(0)<1$. As (A) holds we have that $f_i(0)=0$ for all $i \in \{1,...,n\}$. The Jacobian of $F$ at $0$ can be written as \begin{eqnarray} \label{jacobf0} F'(0) := \left( \begin{array}{cccc}
d_{11}(0)f_1'(0) \  & d_{12}\left( 0 \right)f_{2}'\left( 0 \right) \  & \cdots & d_{1n}\left( 0 \right)f_{n}'\left( 0 \right) \\
d_{21}\left( 0 \right)f_{1}'\left( 0 \right) \ & d_{22}(0)f_2'(0) \ & \cdots \  & d_{2n}\left( 0 \right)f_{n}'\left( 0 \right) \\
\vdots & & \ddots & \vdots \\
d_{n1}\left( 0 \right)f_{1}'\left( 0 \right) \ & d_{n2}\left( 0 \right)f_{2}'\left( 0 \right) \ & \cdots \ & d_{nn}(0)f_n'(0)
\end{array}\  \right).\end{eqnarray} As each $f_i$ is of Kolmogorov type we have that \[f_i'(x) = g_i'(x)x + g_i(x), \ x \geq0 \implies f_i'(0)=g_i(0).\] As $g_i(0)>0$ and $d_{ij}(0) \in (0,1)$, we have that $F'(0)=A(0) \gg 0$. Recall that for any $A := (a_{ij}) \in \mathbb{R}^{n \times n}$ we have that $\rho(A) \leq \|A\|_{1} := \max_{j \in \{1,...,n\}} \left\{\sum_{i =1}^n|a_{ij}| \right\}$ \cite{hj}. Therefore we have that \begin{eqnarray*}
        \rho(A(0)) \leq \|A(0)\|_{1} &=& \max_{i \in \{1,...,n\}} \left\{f_i'(0) \sum_{j =1}^nd_{ji} \right\} \\ &<& \max_{i \in \{1,...,n\}} f_i'(0) \\ &=& \max_{i \in \{1,...,n\}} g_i(0).
    \end{eqnarray*} We thus have that $\rho(A(0))<1$ and so $x^*=0$ is a LAS equilibrium of (\ref{gendisp}), which follows from standard Lyapunov stability results (see for example \cite{hin}). \hfill $\square$
\end{proof}

\begin{remark}
    Note that a sufficient condition for $x^*=0$ to be a LAS equilibrium of (\ref{IVP}) is that $|f'(0)|<1$. This then implies that such a sufficient condition for LAS of extinction for each isolated system, with $f_i \in$ $\cal{M}$, is also sufficient for LAS of the extinction equilibrium for (\ref{gendisp}), by Proposition \ref{max<1}.
\end{remark}

Our next result concerns the global stability of the extinction equilibrium.

\begin{proposition} \label{globalext}
    Suppose $\{f_1,...,f_n\} \subset \mbox{$\cal{M}$}$, (D) holds and $F(x)=A(x)x$, where $A(x)$ is given by (\ref{matrix}). Further assume that the following hold: \begin{enumerate}[(i)]
    \item $d_{ii}(x) \equiv d_i \in (0,1)$;
        \item $g_i(x)$ and $d_{ij}(x)$ are decreasing functions of $x$ for $i \neq j$;
        \item  $\rho(A(0)) < 1$.
    \end{enumerate} Then, $x^*=0$ is a GAS equilibrium of (\ref{gendisp}). 
\end{proposition}
\begin{proof}
Assume $\rho(A(0)) < 1$. In the proof of Proposition \ref{max<1} we saw that $F'(0)=A(0) \gg 0$ and so it follows from Lemma \ref{lem1} that $\mbox{there exists}$ $v^T \gg 0$ such that $v^TA(0) \ll v^T$. Define the function $V:\mathbb{R}^n \rightarrow \mathbb{R}$ as \[V(x) := v^T x.\] As $g_i(x)$ and $d_{ij}(x)$ are decreasing functions of $x$ for $i \neq j,$ and $d_{ii}(x) \equiv d_i \in (0,1)$ we have that \begin{eqnarray*} v^TA(x) \ll v^T A(0) \ \forall \ x>0.\end{eqnarray*} As (\ref{gendisp}) is a positive system, it follows that for $t \in \mathbb{N}$
    \begin{eqnarray*}
        \Delta V := V(x(t+1)) - V(x(t)) &=& V(F(x(t))) - V(x(t)) 
        \\ &=& v^T \left(A(x(t)) - I\right)x(t)  \\ 
        &<& v^T \left(A(0) - I\right)x(t) \\ &<& 0.
    \end{eqnarray*} Hence $V$ is decreasing along non-zero orbits of $F$. Clearly $V$ is positive-definite, i.e. $V(0)=0$ and $V(x)>0$ for all $x >0$. $V$ is radially unbounded, i.e. \[\lim_{\|x\| \rightarrow \infty} V(x) \rightarrow \infty.\] Thus $V$ defines a radially unbounded (copositive) Lyapunov function for $F$ with respect to $x^*=0$ (see for example \cite{hin}). This then implies that \[\lim_{t \rightarrow \infty} F^t(x(t,x_0))=0,\] i.e. $x^*=0$ is GAS for $F$. \hfill $\square$
\end{proof}

\begin{example}
    Examples of $g_i$ and $d_{ij}$ satisfying condition (i) in Theorem \ref{globalext} include \begin{eqnarray*}
        &&g_i(x) = \frac{a_i}{1+b_ix},\mbox{ and } \\
        &&d_{ij}(x) = \exp(-p_{ij}\left(x+q_{ij}\right)),
    \end{eqnarray*} for $a_i, b_i, p_{ij}, q_{ij} >0.$
\end{example}

In \cite{yakubo} they give a sufficient condition for the extinction equilibrium of (\ref{yakubodisp}) to be unstable. We can show that an analogous result holds for (\ref{gendisp}) when regional dynamics are given by maps in $\cal{M}$. First define \[\mbox{$\cal{R}$}_{A(0)} := \min_{i=1,...,n}\left\{g_i(0), \min_{i \in \{1,...,n\}}g_i(0)\sum_{j \neq i}d_{ji}(0) \right\}.\]

\begin{proposition} \label{propRd}
Suppose $\{f_1,...,f_n\} \subset \mbox{$\cal{M}$}$ and $F(x)=A(x)x$, where $A(x)$ is given by (\ref{matrix}). If $\mbox{$\cal{R}$}_{A(0)}>1$ then $x^*=0$ is an unstable equilibrium of (\ref{gendisp}) for all $d_{ij}$ satisfying (D).
\end{proposition}

\begin{proof} Assume $\mbox{$\cal{R}$}_{A(0)} > 1$. In the proof of Proposition \ref{max<1} we saw that $F'(0)=A(0) \gg 0$ and so $\rho(A(0))>1$ is sufficient for instability. Recall that for any $A :=(a_{ij}) \in \mathbb{R}_+^{n \times n}$ one has that $\rho(A) \geq \min_{i \in \{1,...,n\}} \left\{ \sum_{j=1}^na_{ij} \right\}$ \cite{hj}. Therefore it follows that \begin{eqnarray} \label{minrow}
    \rho(A(0)) &\geq& \min_{i \in \{1,...,n\}} \left\{g_i(0)  \sum_{j \neq i} d_{ji}(0)\right\}.
 \end{eqnarray} There are two possible cases to consider. In the first case we have that \[\mbox{$\cal{R}$}_{A(0)} = g_k(0) = \min_{i \in \{1,...,n\}}g_i(0) > 1\] for some $k \in \{1,...,n\}$. In the second case we have that \[\mbox{$\cal{R}$}_{A(0)} = \min_{i \in \{1,...,n\}}g_i(0) \sum_{j\neq i} d_{ji}(0)>1,\] for some $i \in \{1,...,n\}$, this implies that \[1 < \mbox{$\cal{R}$}_{A(0)} \leq \min_{k \in \{1,...,n\}}g_k(0) < g_i(0) \ \forall \ i \in \{1,...,n\}.\] In any case $\min_{i \in \{1,...,n\}}g_i(0)>1$ and so (\ref{minrow}) implies that $\rho(A(0))>1$. \hfill $\square$
\end{proof}

\begin{remark}
   The sufficient conidition assumed in Proposition \ref{propRd} which ensures $\rho(A(0))>1$ is stronger than just assuming $\rho(A(0))>1$. In other words, it is only sufficient for instability. We can find a matrix $A(0) \in \mathbb{R}_+^{n \times n}$ such that $\rho(A(0))>1$ but $\cal{R}$$_{A(0)} \leq 1$. For example, let $n=2$, $d_{12}(0)=d_{21}(0)=0.1$, $d_{11}(0)=0.5$, $d_{22}(0)=0.1$, $g_1(0) = 0.01$ and $g_2(0) = 36$. Then \begin{eqnarray*}
        &&\mbox{$\cal{R}$$_{A(0)}$} = g_1(0)d_{21}(0) \approx 0.001 < 1 \\ 
        &&\rho(A(0)) \approx 3.6 > 1.
    \end{eqnarray*} 
\end{remark}

We now show that $F$ is strongly positive.

\begin{lemma} \label{pos} Suppose $\{f_1,...,f_n\} \subset \mbox{$\cal{M}$}$, (D) holds and $F(x)=A(x)x$, where $A(x)$ is given by (\ref{matrix}). Then, $F(x) \gg 0$ for all $x \in \mathbb{R}_+^n \backslash \{0\}$ and for all $d_{ij}$ satisfying (D).
\end{lemma}

\begin{proof}
As (B) and (D) hold, clearly $A(x) \gg 0 \implies F(x) \gg 0$ for all $x >0$. \hfill $\square$
\end{proof}


Our next result establishes the boundedness of the map $F$, which in turn implies that the associated system is point-dissipative in the terminology of \cite{st}. 

\begin{proposition} \label{pointdis} Suppose $\{f_1,...,f_n\} \subset \mbox{$\cal{M}$}$, (D) holds and $F(x)=A(x)x$, where $A(x)$ is given by (\ref{matrix}). Then, $\exists$ \(M > 0\) such that \mbox{$\forall$} \(x \in \mathbb{R}_+^{n}\) \[\|F(x)\|_1 \leq M.\] 
\end{proposition}

\begin{proof}

For $x \in \mathbb{R}_+^n$, the $l_1$ norm of $x$ is given by \[\|x\|_1 = \sum_{i=1}^n x_i = \mathbbm{1}^T x,\] where $\mathbbm{1} := (1,...,1)^T$. It follows from Lemma \ref{pos} that $F(x) \gg 0$ $\forall$ $x \geq0$. As (D) holds we therefore have that \begin{eqnarray*} \|F(x)\|_1 = \mathbbm{1}^T F(x) &=& \sum_{i=1}^n f_i(x_i) \leq \sum_{i=1}^n m_i, \ \forall \ t>0. \end{eqnarray*} Thus define the upper bound $0<M := \sum_i^n m_i$. \hfill $\square$
\end{proof}

\begin{remark} By Lemma \ref{pos} and Proposition \ref{pointdis} the cone $\mathbb{R}_+^n$ is forward invariant under $F$ and the sequence $\{F^t(x)\}_{t \geq 0}$ has no unbounded orbit.  It is also immediate that for any $x_0 \in \mathbb{R}^n_+$, $\|x(t,x_0)\|_1 \leq M$ for all $t \geq 1$. \end{remark}

We now state our first main result, which gives a sufficient condition for the existence of a positive fixed point of $F$.

\begin{theorem}
\label{fix} Suppose $\{f_1,...,f_n\} \subset \mbox{$\cal{M}$}$, (D) holds and $F(x)=A(x)x$, where $A(x)$ is given by (\ref{matrix}). If $\rho(A(0)) > 1$, then $\exists$ $x^* \in$ Int$(\mathbb{R}_+^n)$ such that $F(x^*)=x^*$. 
\end{theorem}

\begin{proof}
It follows from Proposition \ref{pointdis} that $\mbox{there exists}$ a constant $M>0$ such that $\|F(x)\|_1 \leq M$ for all $x \in \mathbb{R}_+^n$. If $\rho(A(0))>1$, it follows from Lemma \ref{lem1} that $\exists$ $v \gg 0$ such that \[v^TA(0) \gg v^T.\]  As $A:\mathbb{R}_+^n \rightarrow \mathbb{R}_+^{n \times n}$ is a continuous matrix-valued function $\mbox{there exists}$ a constant $0<\delta < M$ such that \[v^TA(x) \geq v^T, \ \|x\|_1 \leq \delta.\] For any such $x \geq 0$ ($x \neq 0$) we then can see that \[v^TF(x) = v^TA(x)x \geq v^Tx.\] 

 If we define \begin{eqnarray*}
    \Omega_0 &:=& \{ x \in \mathbb{R}_+^n : \delta \leq \|x\|_1 \leq M \}
\end{eqnarray*} we then have that $F(x) \gg 0$ for $x\in \Omega_0$ and hence \[x \in \Omega_0 \implies  v^TF(x)>0.\] As $A$ is continuous, $F$ is also continuous. It then follows from the Extreme Value Theorem that $v^T F(x)$ attains its minimum and maximum when $\delta \leq \|x\|_1 \leq M$. Therefore $\mbox{there exists}$ $\kappa_1>0$ such that \[\min \left\{v^TF(x), \ x \in \Omega_0\right\} = \kappa_1.\] Now choose some $\bar{x} >0 $ with $\|\bar{x}\|_1 \leq M$. Then, $\exists$ $\kappa_2>0$ such that $v^T\bar{x} = \kappa_2$. Let $\kappa = \min\{\kappa_1, \kappa_2\}$ and define the region \[\Omega_1 := \{ x \in \mathbb{R}_+^n : v^Tx \geq \kappa, \ \|x\|_1 \leq M \}.\] By construction $\Omega_1$ is non-empty. Clearly $\Omega_1$ is closed and bounded, so therefore compact. Let $x, y \in \Omega_1$ and let $\alpha \in [0,1]$ be arbitrary. Then, we have that \begin{eqnarray*} v^T((1-\alpha) x + \alpha y) \geq (1-\alpha) \kappa_1 + \alpha \kappa_1 = \kappa_1, \end{eqnarray*} and by the triangle inequality \begin{eqnarray*}\|(1-\alpha) x + \alpha y\| &\leq& (1-\alpha)\|x\| + \alpha\|y\| \\ &\leq& (1-\alpha) M + \alpha M \\ &=& M.\end{eqnarray*} Therefore \[x,y \in \Omega_1 \implies (1-\alpha) x + \alpha y \in \Omega_1 \ \forall \ \alpha \in [0,1],\] i.e. $\Omega_1$ is convex. Let $z \in \Omega_1$. We have already shown that \[\delta \leq \|z\|_1 \leq M \implies v^TF(z) \geq \kappa_1 \implies F(z) \in \Omega_1.\] Since $\delta \leq M$ we must also have that \[\|z\|_1 \leq \delta \implies v^TF(z) \geq v^Tz \geq \kappa_1 \implies F(z) \in \Omega_1.\] In any case \[x \in \Omega_1 \implies F(x) \in \Omega_1.\] Hence we have found a non-empty, compact, convex set in $\mathbb{R}_+^n \backslash\{0\}$, given by $\Omega_1$, such that $F$ maps $\Omega_1$ into itself. Therefore, as $F$ is continuous on $\Omega_1$, it follows from Brouwers Fixed Point Theorem that \[\exists \ x^* \in \Omega_1 : F(x^*)=x^*.\] Lemma \ref{pos} ensures that $x^* \not \in \partial \mathbb{R}_+^n$, i.e. $x^* \gg 0$. \hfill $\square$


\end{proof}

\begin{remark}
It should be noted that Theorem 7.5 of \cite{st} could be used to give an alternative proof of the last result.  However, our proof above directly uses the specific properties of the system (\ref{gendisp}) rather than relying on more general conditions.  Moreover, it clarifies that the positive equilibrium is contained in the region $\Omega_1$.  Similarly to Theorem 7.5 of \cite{st}, we have not established uniqueness or stability of the positive equilibrium.  We shall describe several numerical investigations later to give some insight into these and related questions.  
\end{remark}





%


\subsection{Persistence}

An alternative perspective to stability is that of uniform persistence. Note that as $F'(0) = A(0)$ is irreducible (as shown in Proposition \ref{max<1}), it follows from Corollary 3.18 of \cite{st} that if $\rho(A(0))>1$, then (\ref{matrixmodel}) is uniformly weakly $\|\cdot\|$-persistent for $\|\cdot\|$ any norm on $\mathbb{R}^n$. We will now show that (\ref{gendisp}) is also uniformly strongly persistent with respect to two different persistence functions.

\begin{theorem} \label{persthem} Suppose $\{f_1,...,f_n\} \subset \mbox{$\cal{M}$}$ and $F(x)=A(x)x$, where $A(x)$ is given by (\ref{matrix}). Let $\|\cdot\|$ be any norm on $\mathbb{R}^n$. Define $\eta_1(x) := \min_{i \in \{1,...,n\}} x_i$ and $\eta_2(x) := \|x\|$. If $\rho(A(0)) > 1$, then $F$ is uniformly strongly $\eta_i$-persistent, $i \in \{1,2\}$, for all $d_{ij}$ satisfying (D).\end{theorem}
\begin{proof}
Assume $\rho(A(0))>1$. It follows from Lemma \ref{pos} that $F(x) \gg 0$ for all $x > 0$ and so \[F\left(\mathbb{R}_+^n \backslash \{0\}\right) \subset \mbox{Int}\left(\mathbb{R}_+^n\right).\] Clearly $\rho(A(0))>1$ implies that $\rho(A(0)^T)>1$.  In the proof of Proposition \ref{max<1} we saw that $F'(0)=A(0) \gg 0$. It then follows from Lemma \ref{lem1} that $\mbox{there exists}$ $v \gg 0$ such that \[F'(0)^Tv = A(0)^Tv \gg v \implies F'(0)^Tv \geq r_0v\] for some suitably chosen $r_0>1$. It follows from Proposition \ref{pointdis} that $\mbox{there exists}$ $M>0$ such that $\|F(x)\|_1 \leq M$ for all $x>0$. As noted in \cite{st} the set \[X_{0}^{(k)} := \{x_{0} \in \mathbb{R}_{+}^{n} : \eta_k \left(x(t,x_{0}) \right) = 0, \ \forall \ t \geq 0\}\] is equal to $\{0\}$ precisely when $F(0)=0$ and for all $c>0$ there exists $s>0$ such that $F^s \gg 0$ when $0 <\|x\|\leq c$, $k \in \{1,2\}$ ($\|\cdot\|$ any norm on $\mathbb{R}^n$). As (A) holds we have that \[f_i(0)=0 \ \forall \ i \in \{1,...,n\} \implies F(0)=0.\] Therefore we have that $X_0^{(k)} = \{0\}, \ k\in \{1,2\}$ (with $s=1$). We have thus verified the assumptions of Theorem \ref{stthm} (with $H=F$). We can thus conclude that $\mbox{there exists}$ $\epsilon_1, \epsilon_2 > 0$ such that \begin{eqnarray*} && \min_{i \in \{1,...,n\}}x_i(0) > 0 \implies \liminf_{t\rightarrow \infty} \min_{i \in \{1,...,n\}} x_i(t) \geq \epsilon_1, \\  && \|x(0)\|>0 \implies \liminf_{t\rightarrow \infty} \|x(t)\| > \epsilon_2, \end{eqnarray*} i.e. (\ref{gendisp}) is uniformly strongly $\eta_i$-persistent, $i \in \{1,2\}$. \hfill $\square$
\end{proof}

\section{Numerical Results}

We now use numerical methods to investigate several questions suggested by the results of the previous section. Let us consider the simplest case of (\ref{gendisp}) when $n=2$. This means that (\ref{gendisp}) reduces to \begin{equation} \label{2by2model}
\begin{aligned} 
    &x_1(t+1) = d_{11}(x_1(t))f_1(x_1(t)) + d_{12}(x_2(t))f_{2}(x_2(t)) \\
    &x_2(t+1) = d_{22}(x_2(t))f_2(x_2(t)) + d_{21}(x_1(t))f_{1}(x_1(t)), \\
    &x(0) \in \mathbb{R}_+^2.
\end{aligned}
\end{equation}

If we consider (\ref{2by2model}) with each $f_i$ given by a Generalised Beverton-Holt, Hassell, Ricker or Logistic map. Then, if $a_i \in (0,1)$ for $i=1,2$, it follows from Proposition \ref{max<1} that extinction for (\ref{2by2model}) is at least LAS for any choice of $D=\left(d_{ij}\right)$. 

\vspace{3mm}

Throughout the rest of this section we will show how our model can account for various source-sink type dynamics. All computations in this Section were conducted using R (version 4.3.1) \cite{Rcore}. 

\subsection{Sources and Sinks}

Given a $C^1$ Kolmogorov type map, $f(x)=g(x)x$, it is known that \cite{schreiber03} \begin{enumerate}[(\mbox{S}1)]
\item $g(x)<1 \ \forall \ x \geq 0 \implies x(t) \rightarrow 0$ as $t \rightarrow \infty$ for all $x(0) > 0$, and 
\item $g(0)>1 \implies x(t) \not \rightarrow 0$ as $t  \rightarrow \infty$ for some $x(0)>0$.
\end{enumerate} For $f_i \in$ $\cal{M}$, we call region $i$ a \textit{sink} if it is a low quality habitat where an isolated population would go extinct, i.e. S1 holds for $g_i$. On the other hand, region $i$ is a \textit{source} if it can sustain an inhabiting population, i.e. S2 holds for $g_i$. 

\vspace{3mm}

Unless stated otherwise, we assume that $f_1$ is given by a Ricker map, i.e. \[f_1(x)= a_1x\exp(-b_1x)\] and $f_2$ is given by a Hassell-1 map (a Hassell map as in Remark 1 with $c=1$), i.e. \[f_2(x)=\dfrac{a_2x}{1+b_2x},\] where $a_i, b_i>0$ for $i=1,2$. For Ricker and Hassell-1 maps one has that \[g_i(0)=a_i<1 \implies g_i(x_i) < 1 \ \forall \ x \geq 0\] and so the extinction equilibrium is GAS for $f_i$.

\vspace{3mm}
    
If $a_i>1$, the extinction equilibrium is unstable for $f_i$. This also implies that there exists a unique positive equilibrium for $f_i$. For $f_1$ this is given by \[x_R^* := \frac{\ln(a_1)}{b_1}\] and for $f_2$ this is given by \[x_H^* := \frac{a_2-1}{b_2}.\] For the Ricker map, if $1<a_1<e$ we have that $x_R^*$ is LAS. If $a_1 > e$ we have that $x_R^*$ is unstable. For the Hassell-1 map, if $a_2 > 1$ we have that $x_H^*$ is GAS \cite{BHevol}.

\vspace{3mm}

Unless stated otherwise we further assume that region $1$ is a source (S1) and region $2$ is a sink (S2), i.e. $a_1 >1$ and $a_2<1$. Thus we are interested in examining some of the dynamical behaviour of (\ref{2by2model}) in various source-sink dispersal scenarios.

\subsection{Density-Dependent Dispersal}
Unless stated otherwise, throughout the rest of this section, let the dispersal functions in (\ref{2by2model}) be given by \begin{eqnarray}
\label{disp_example} &&d_{ij}(x_j) = \frac{r_{ij}}{(1+\exp\left(-k_{ij}\left(x_j-s_{ij}\right)\right)},\end{eqnarray} where $r_{ij} \in (0,1)$ and $k_{ij}, s_{ij} \geq 0$ for $i,j \in \{1,2\}$. These dispersal functions may describe, for example, regional populations that exhibit a mixture of both \textit{negative} and \textit{positive} density-dependent (DD) dispersal \cite{rodrigues}. Positive (negative) DD dispersal is where, for large densities, the proportion of individuals dispersing from (remaining on) a region is high, which can be controlled through the parameters $r_{ij}$, $k_{ij}$ and $s_{ij}$. For example, positive DD dispersal may be due to factors such as \textit{competition}/\textit{crowding} or \textit{dominance hierarchies} \cite{matthysen}. On the other hand, negative density dependence may be due to factors such as \textit{aggressive interactions}, for example. 

\vspace{3mm}

For instance, one can model positive DD dispersal on both regions. For example, let both $r_{ii}$ and $k_{ii}$ be sufficiently small, and $r_{ji}$ and $k_{ji}$ be sufficiently large, for $i \neq j$. In other words, $d_{ii}(x_i)$, which models the proportion of individuals remaining on region $i$, will be relatively small for all values of $x_i$ and is an increasing function of $x_i$, with the rate of increase being quite gradual. At the same time, $d_{ji}$ rapdily increases toward $r_{ji}$, which is relatively close to $1$, as $x_j$ increases.

\begin{figure}[!h]
\centering
\includegraphics[trim = 200 80 200 80, clip, width = 11cm]{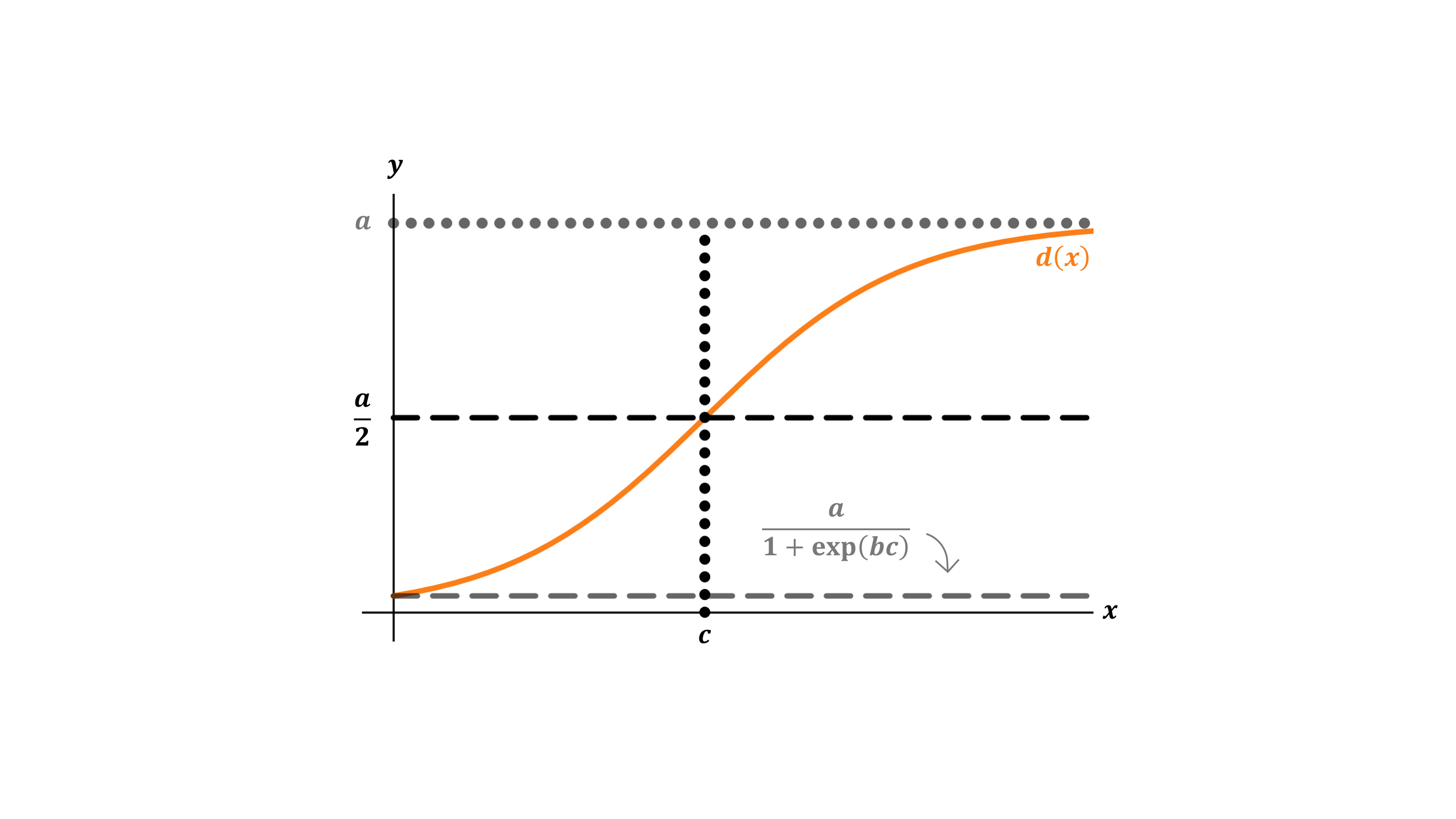}
\caption{ Illustration of $y = d(x) := a(1+\exp\left(-b\left(x-c\right)\right)^{-1}$ (orange solid curve) for $a \in (0,1)$ and $b,c \geq 0$. The dotted grey line is the curve $y=a$. The dashed grey line is the curve $y=d(0)=a(1+\exp(bc))^{-1}$. The dashed black line is the curve $y=d(c)=a/2$. The dotted black line is the curve $x=c$.}
\label{disp_image}
\end{figure}

 \begin{remark} The dispersal function (\ref{disp_example}) is also a specific form of the \textit{generalised logistic function} or so-called \textit{Richards curve} \cite{richards}. An illustration of a function of the form of (\ref{disp_example}) is given in Fig. \ref{disp_image}. We can observe that the following hold: \begin{eqnarray*}
    && d_{ii}(x)+d_{ji}(x) < 1 \iff r_{ii} + r_{ji} < 1, i \neq j; \\
    &&d_{ij}(0) = r_{ij}(1+\exp(k_{ij}s_{ij}))^{-1}>0; \\
    &&\lim_{x \rightarrow \infty} d_{ij}(x) = r_{ij}<1; \mbox{ and } \\
    &&d_{ij}\left(s_{ij}\right) = \frac{r_{ij}}{2}.
 \end{eqnarray*}\end{remark}

\begin{figure}[!h]
\centering
\includegraphics[trim = 5 30 5 30, clip, width = 12cm]{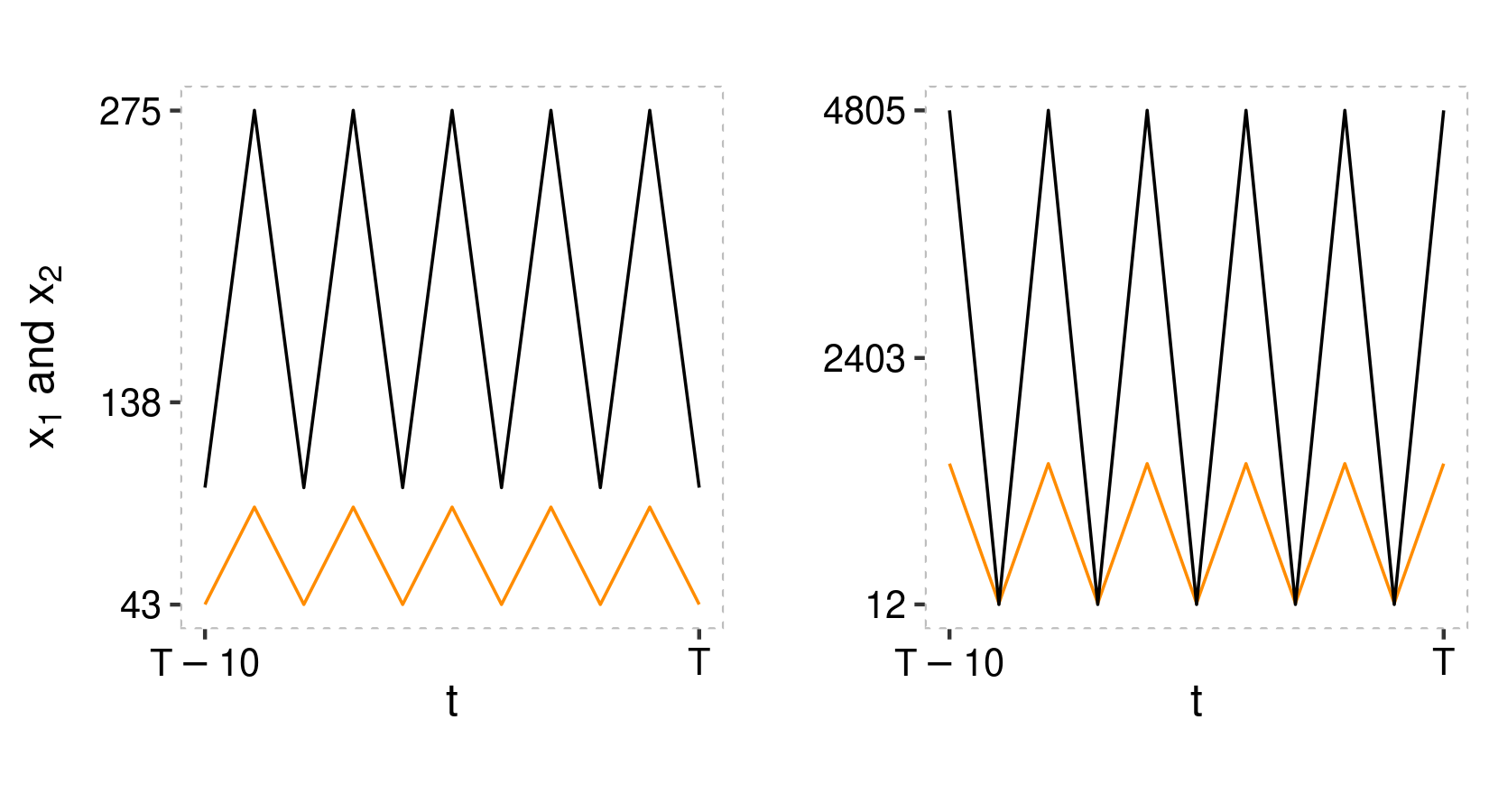}
\caption{Simulated periodic dynamics of (\ref{2by2model}) for (left) $a_1=50$, i.e. (A1) and (right) $a_1=750$, i.e. (A2). We plot the last $10$ out of $T=500,000$ time steps. Here, orange is region $1$ and black is region $2$. In both simulations we let $a_2 = 0.4, b_1=0.04, b_2=0.01, r_{11}= 0.2, r_{22}  = 0.3, r_{12}  = 0.6, r_{21}  = 0.7, k_{11}=k_{22}=k_{12}=k_{21}=0.5, s_{11}=10, s_{22}=6, s_{12}=3$ and $s_{21}=12$. Initial conditions in both were $(x_1(0), x_2(0))=(92, 103)$.}
\label{xplots}
\end{figure}

\subsection{Stability Dichotomy}

An obvious question to ask, in relation to (\ref{gendisp}), is: does an analagous result hold to Theorem \ref{kirkthm}? That is, for $\{f_1, ..., f_n\} \subset$ $\cal{M}$,
\begin{enumerate}[(Q1)]
    \item does $\rho(A(0))<1$ $\implies$ $x^*=0$ is a GAS equilibrium for system  (\ref{gendisp}); and
    \item does $\rho(A(0))>1$ $\implies$ $x^*\gg 0$ is a GAS equilibrium for system  (\ref{gendisp})?
\end{enumerate} The answers to (Q1) and (Q2) are in fact both no, as we can demonstrate using an example.

\begin{example} Consider (\ref{2by2model}) where regions $1$ and $2$ are respectively given by Ricker and Hassell-1 maps. Let all parameters be as in Fig. \ref{xplots}. We can show that (Q1) and (Q2) do not hold in general for appropriate choices of $a_1$:  

\begin{enumerate}[(\mbox{A}1)]
    \item Let $a_1=50$. Then we get that $\rho(A(0)) \approx 0.1051 < 1$. If we simulate this parameterised system we can observe in Fig. \ref{xplots} (left) the existence of a periodic solution of period $2$. Hence $x^*=0$ is not GAS, i.e. (Q1) does not hold.

    \item Let $a_1=750$. Then we get that $\rho(A(0)) \approx 1.058 > 1$. If we simulate this parameterised system we can observe in Fig. \ref{xplots} (right) the existence of a periodic solution of period $2$. Hence $x^* \gg 0$, which exists by Theorem \ref{fix}, is not GAS, i.e. (Q2) does not hold.
\end{enumerate}

In (A1) and (A2) above we checked that \[F^2\left(x_P^{(i)}\right)=F\left(F\left(x_P^{(i)}\right)\right)=x_P^{(i)}\] where $x_P^{(i)}$ is the period-$2$ limit cycle in scenario ($A_i$).
\end{example}

\begin{figure}[!h]
\centering
\includegraphics[trim = 5 15 5 15, clip, width = 12cm]{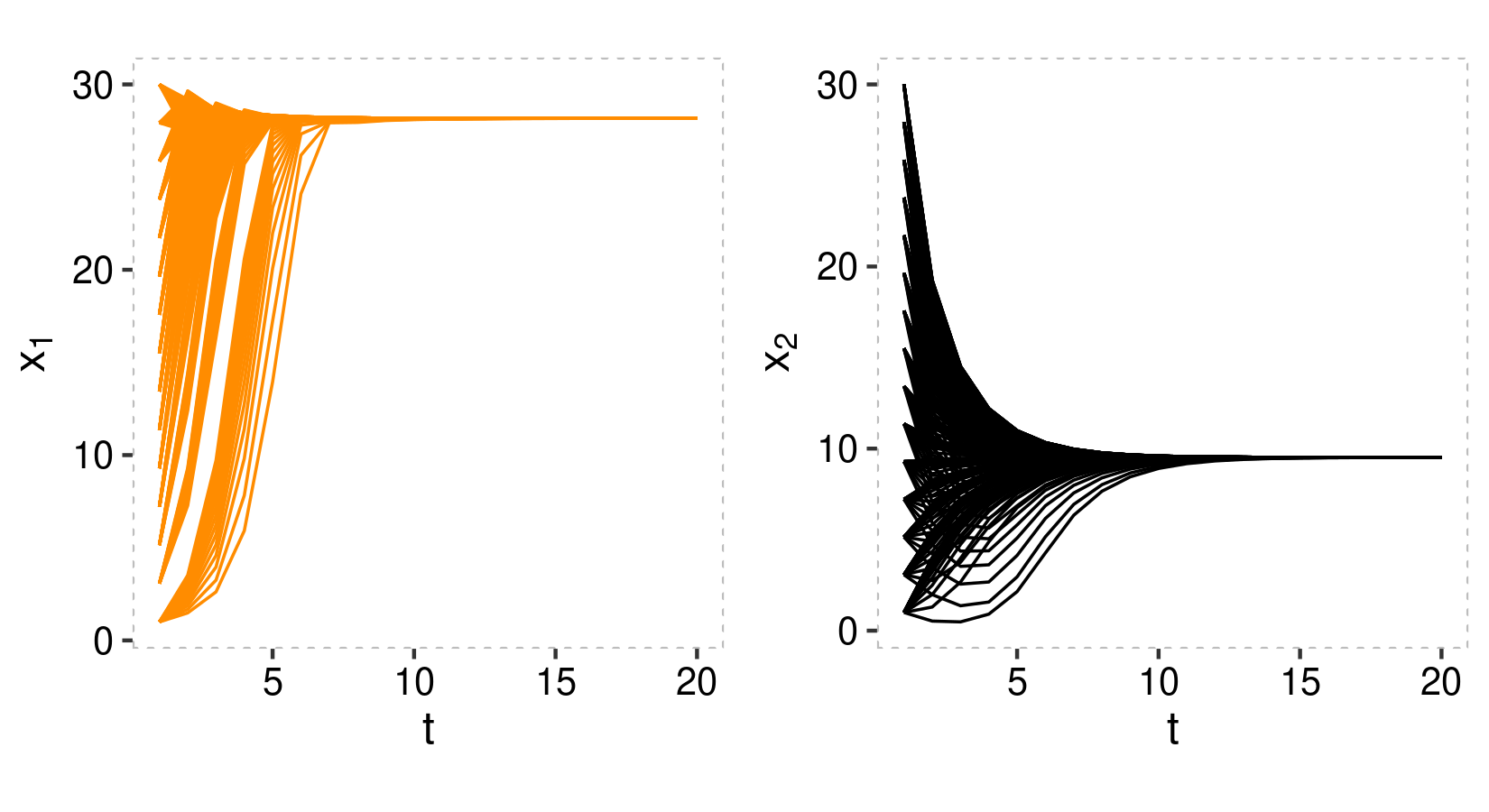}
\caption{ Simulations of (\ref{2by2model}) for various positive initial conditions $x(0) \in [1,30]^2$, where orange (left) and black (right) trajectories respectively correspond to regions $1$ and $2$. We let $a_1=4$, $a_2=0.9$, $b_1=0.04$, $b_2=0.01$, $r_{12}=r_{21}=0.1$, $r_{11} = r_{22} = 0.75$, $k_{12}=k_{21}=k_{11}=k_{22}=1$ and $s_{12}=s_{21}=s_{11}=s_{22}=1$. We increased the number of iterations for each plot to $100,000$ to ensure convergence to the fixed point approximately given by $(x_1^*, x_2^*) \approx (28.17, 9.52)$.}
\label{xplot_initial}
\end{figure}

\subsection{Regional vs Coupled Dynamics}

Consider (\ref{2by2model}) where regions $1$ and $2$ are respectively given by Ricker and Hassell-1 maps. One may ask, is coupling a source in $\cal{M}$, with a LAS positive equilibrium, to a sink in $\cal{M}$, sufficient for $\rho(A(0))>1$? The answer is in fact no. We show this in the following example. \begin{example} Let $a_1 = 1.1$ and $a_2=0.6$. Recall that for $f_1$, if $1<a_1<e$ then $x_R$ is LAS. Further let all other parameters be as in Fig. \ref{xplot_initial}. We then have that $\rho(A(0)) \approx 0.226 <1$. Here $r_{ii}>r_{ji}$ (for $i \in \{1,2\}$ and $i \neq j$) and so at high regional densities we observe a higher proportion of individuals remaining on patch $i$ than migrating, which could, for example, signify \textit{natal philopatry}, the tendency for individuals to stay in the region they are born \cite{matthysen}.\end{example}

Theorem \ref{fix} shows that $\rho(A(0))>1$ is sufficient for the existence of a positive fixed point for (\ref{gendisp}). In the next example, we show that this is not necessary. 

\begin{example} Let $a_1=4$ and $a_2=0.9$. Further let all other parameters be as in Example 4. We can then observe that $\rho(A(0)) \approx 0.81 < 1$ but $(x_1^*, x_2^*) \approx (28.17, 9.52)$ is a positive equilibrium of this system (see Fig. \ref{xplot_initial}). We simulated this system for $400$ different positive initial conditions in Fig. \ref{xplot_initial}. The results of this simulation suggest this equilibrium is unique and locally stable. We observed convergence to the extinction equilibrium when initial conditions were sufficiently close to $\partial \mathbb{R}_+^2$. This behaviour is distinct to the dynamical behaviour in \cite{kirkland}, as seen via Theorem \ref{kirkthm}. As mentioned above, the spectral radius dichotomy for stability does not, in general, hold for the model class we consider. An interesting line of future research would be to further understand under what conditions, on both $f_i \in$ $\cal{M}$ and $d_{ij}$ satisfying (D), result in this dichotomous condition holding. \end{example}

\begin{figure}[!h]
\centering
\includegraphics[trim = 5 0 5 0, clip, width = 12cm]{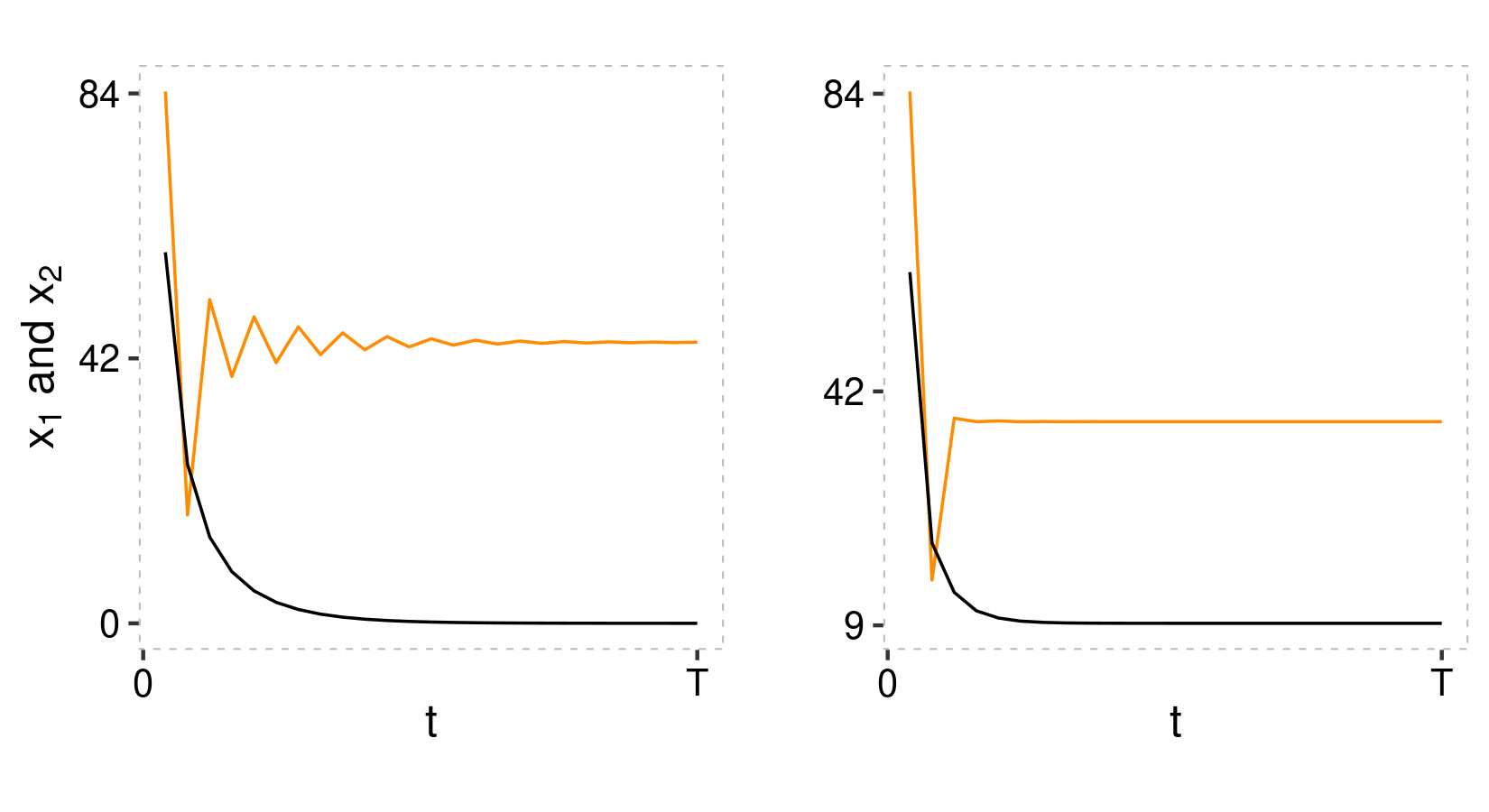}
\caption{A plot of $x_1$ (orange) and $x_2$ (black) when patches are isolated (left) and coupled (right) for $T=20$ time steps. We let $a_1= 5.94$ and $a_2=0.68$. All other parameters were as in Fig. \ref{xplot_initial}. Initial conditions, $(x_1(0), x_2(0))=(84, 59)$.}
\label{iso_coup}
\end{figure}

\subsection{The Rescue Effect and Total Population Size}
By studying a low-dimensional model such as (\ref{2by2model}), one can also show how the rescue effect can arise from coupling sinks to sources, i.e. preventing the declining sink region from going extinct via dispersal \cite{brown}. We will demonstrate this phenomenon using a simple example. \begin{example} Let $a_1=5.94$ and $a_2=0.68$. Further let all other parameters be as in Fig. 3. We can then observe that $\rho(A(0)) \approx 1.2 > 1$. This implies that (\ref{2by2model}) has a coexistence equilibrium by Proposition \ref{fix}. By Theorem \ref{persthem}, $F$ is also uniformly strongly persistent with respect to the total population size, $\|x\|_1$, the maximum sized region, $\|x\|_{\infty}$, and minimum sized region, $\min_{i \in \{1,2\}}x_i$. It thus follows from Theorem \ref{persthem} that, even though on one patch extinction may be inevitable, by allowing dispersal one can not only rescue a sink, but also ensure that each regional population persists. In Fig. \ref{iso_coup} we can observe the dynamics of our isolated system and (\ref{2by2model}) for this particular paramterisation. In this case we can see that trajectories for (\ref{2by2model}) tend to a positive fixed point. \end{example}

In \cite{franco} the authors looked at how dispersal affects the total population before and after coupling. They demonstrated that for the specific planar systems with constant dispersal rates, dispersal can have varying effects on the overall population. We will now present a simple example where one can observe analogous effects, when dispersal is nonlinear and asymmetric, and when regional maps come from different model classes. 

\begin{example} To demonstrate these effects, in the spirit of parsimony, we will fix all parameters except the \textit{dispersal saturation} parameters, $r_{12}$ and $r_{21}$. In particular we set $r_{12}=r_{21}=r \in (0,0.5)$. Let $a_1=65$ and $a_2=0.4$. We let $d_{ii}$ to be constant, i.e. the proportion remaining in region $i$ stays fixed over time. In particular let $d_{11}(x)=d_{22}(x) \equiv d = 0.499$. 
Further let the remaining parameters be as in Example 4. We then have that \begin{eqnarray*} A(0) = \left( \begin{array}{cc} da_1  & \dfrac{r a_2}{1+\exp(k_{12}s_{12})} \\
\dfrac{r a_1}{1+\exp(k_{21}s_{21})} & da_2\end{array}\  \right),\end{eqnarray*} which, for a known $r$, has characteristic polynomial approximately given by \[ \chi(\lambda) \approx \left(32.4350-\lambda \right)\left(0.1996- \lambda \right)-1.8806r^2.\]

\begin{figure}[!h]
\centering
\includegraphics[trim = 5 0 5 0, clip, width = 11cm]{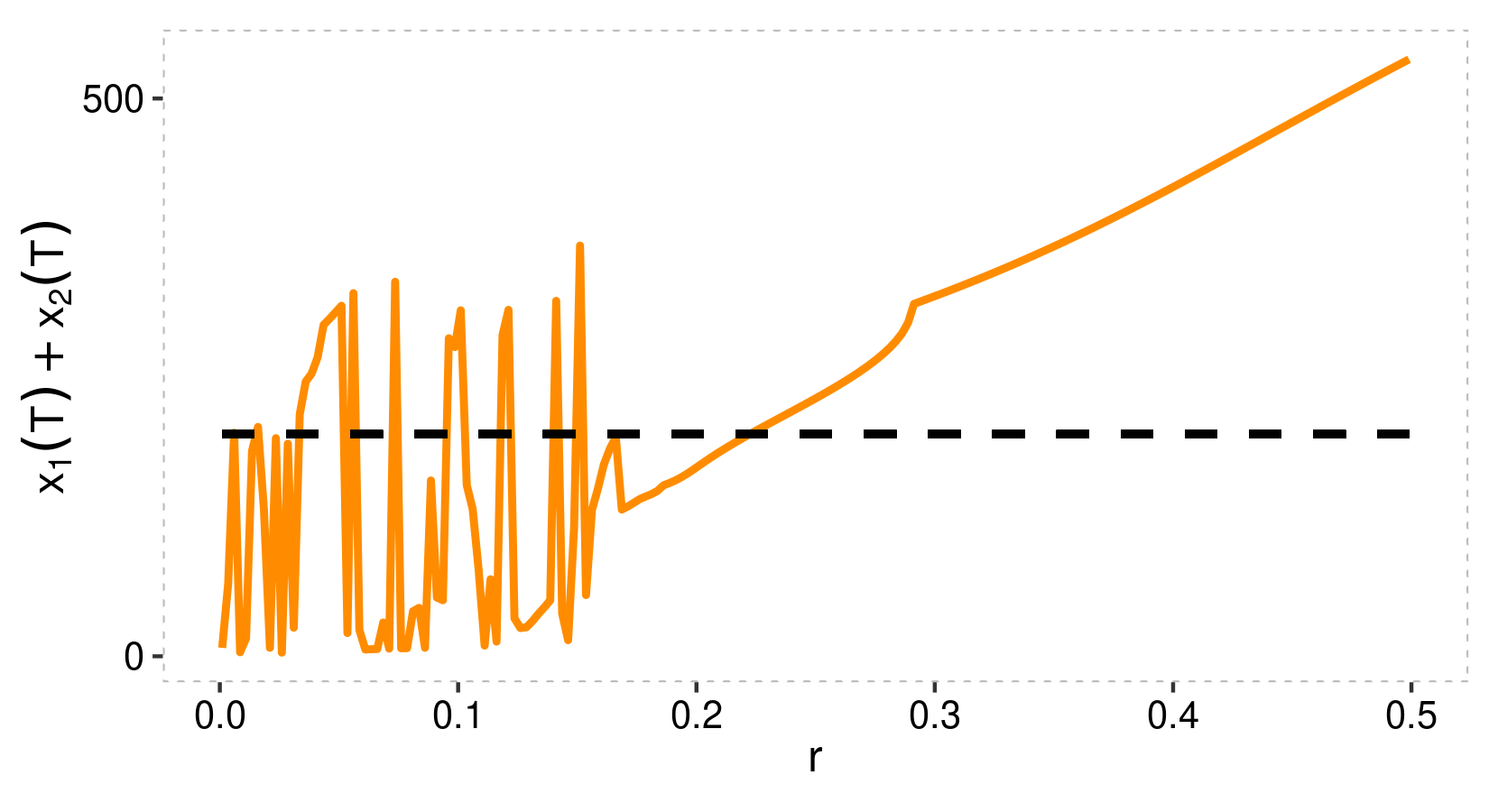}
\caption{ A plot of $x_1(T)+x_2(T)$ when patches are connected by dispersal (solid orange) and when isolated (dashed black) after $T=10,000$ time steps. Initial conditions, $(x_1(0), x_2(0))=(55, 54)$.}
\label{disp_total}
\end{figure}

It is then clear to see, by monotonicity of $\sqrt{\cdot}$, that \[\rho(A(0)) \approx  da_1 \approx 32.4 > 1 \ \forall \ r \in (0, 0.5).\]

By Theorem \ref{persthem}, $F$ is uniformly strongly persistent with respect to the total population size. Despite this we can see from Fig. \ref{disp_total} that for certain values of $r$, dispersal can have an increasing, decreasing or neutral effect on the total population size. In our example we see from Fig. \ref{disp_total}, that after some critical value, $r_{c} \approx 0.22$, increasing $r$ has a monotonically increasing effect on the total population size in this context.

\vspace{3mm}

The importance of studying the effect of dispersal on the total population size has been emphasised in \cite{grombach} for planar systems, where the authors assumed each region was a source and dispersal was constant. We have demonstrated, as in \cite{franco} for the case of constant dispersal, that the manner in which a source is coupled to a sink can impact the total population size. 

\vspace{3mm}

We also conducted an analagous simulation for when both regions were sources (see Fig. \ref{sourcepic}). We kept the initial conditions and all parameters the same as in Fig. \ref{disp_total} and set $a_2=10$, i.e. region $2$ is also a source. We can then see that, although $a_1$ is the same as in the previous source-sink case, when we connect two sources by dispersal we can get that the total population size is always below that of when the regions are isolated, for identical initial conditions. An interesting prospect for future work would be to extend the results of \cite{grombach} to investigate the varying cases for total population size when one looks at dispersal (i) only between sources and (ii) between sources and sinks. 
\end{example}

\begin{figure}[!hbt]
\centering
\includegraphics[trim = 5 0 5 0, clip, width = 11cm]{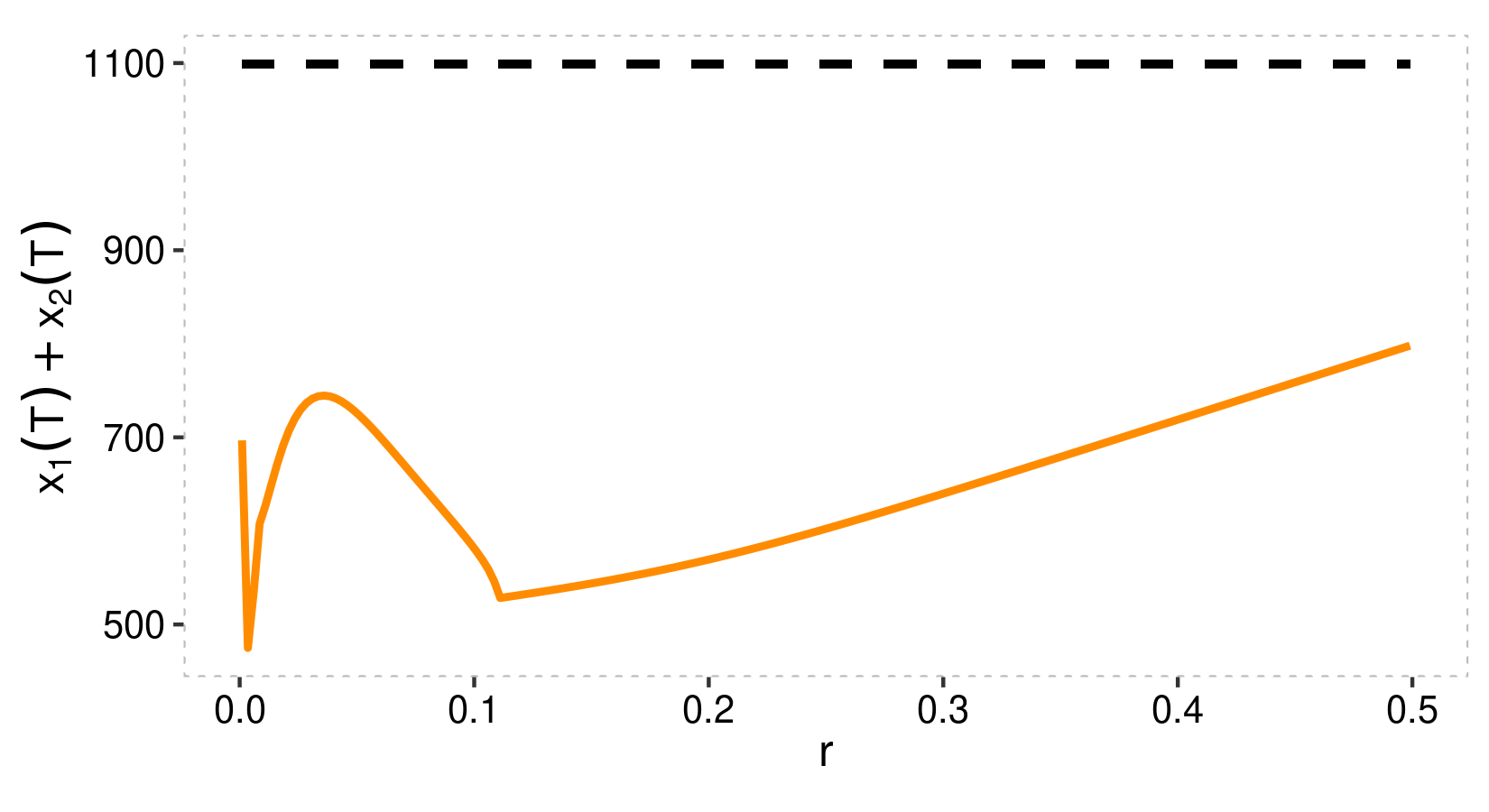}
\caption{ A plot of $x_1(T)+x_2(T)$ when two source regions are connected by dispersal (solid orange) and when isolated (dashed black) after $T=10,000$ time steps. Initial conditions, $(x_1(0), x_2(0))=(55, 54)$, were as in Fig. \ref{disp_total}.}
\label{sourcepic}
\end{figure}

\subsection{Bifurcation Analyses}

In the above numerical simulations we observed that varying dispersal and regional parameter values can result in quite complex dynamical behaviour for trajectories of system (\ref{2by2model}). Therefore in order to briefly explore parameter sensitivity we will conclude this section by looking at single- and two-parameter bifurcation diagrams.

\begin{figure}[!h]
\centering
\includegraphics[trim = 50 5 50 5, clip, width = 12cm]{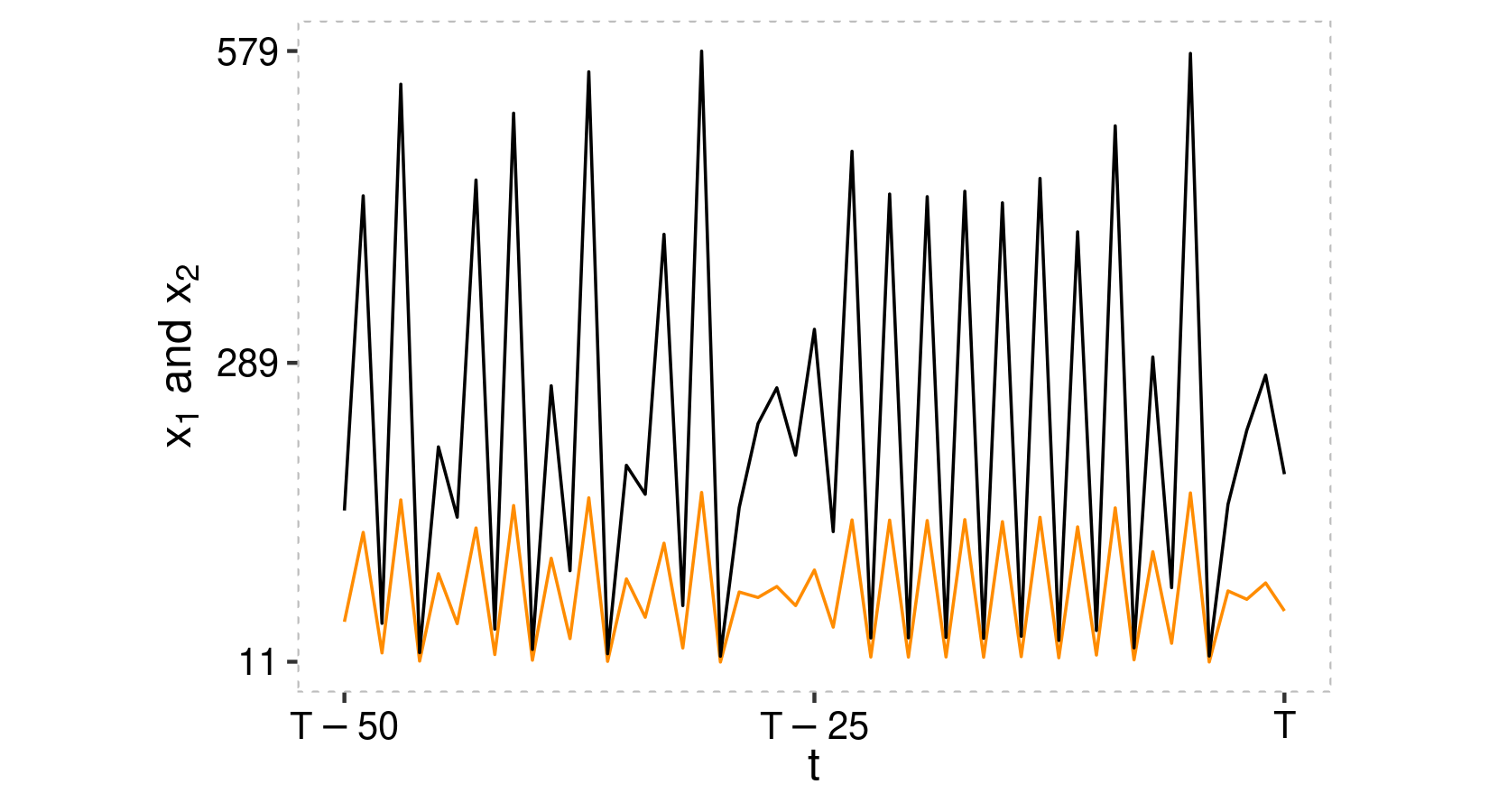}
\caption{Simulated dynamics of (\ref{2by2model}) for the initial condition $(x_1(0), x_2(0))=(131, 19)$, where orange and black trajectories respectively correspond to region $1$ and $2$, where $T=500,000$. We let $a_1=90$, $a_2=0.14$, while all other parameters where as in Fig. \ref{xplots}. The last $50$ time steps are plotted.}
\label{timeseries}
\end{figure}

As $g_i(0) := a_i$ determines stability (of the extinction and/or positive equilibrium) on region $i$, we produced a two-parameter bifurcation diagram for $a_1$ and $a_2$, to observe what range of dynamical behaviour can arise. 

\vspace{3mm}

In \cite{strogatz}, the author informally describes chaotic behaviour as ``\textit{aperiodic long-term behaviour in a deterministic system that exhibits sensitive dependence on initial conditions}". Fig. \ref{timeseries} shows the results of a numerical simulation of \eqref{2by2model} with $a_1=90$ and $a_2=0.14$, and all other parameters as in Fig. \ref{xplots}. We simulated (\ref{2by2model}) for $T = 500,000$ time steps, while also computing the number of unique values attained over $[0,T]$. This resulted in no observed periodic behaviour. We also tested if this aperiodicity is sensitive to choices in initial conditions. In particular, for sufficiently small $\epsilon_i>0$, initial conditions given by $(x_1(0), x_2(0))=(131 \pm \epsilon_1, 19 \pm \epsilon_2)$, resulted in trajectories of (\ref{2by2model}) that enter a different aperiodic regime. These observations suggest that \eqref{2by2model} may be capable of dynamical behaviour similar to that exhibited by chaotic systems. Of course, this is not a fully rigorous analysis and more work is required before a definitive statement on this can be given.

\begin{figure}[!hbt]
\centering
\includegraphics[trim = 0 8 0 8, clip, width = 13cm]{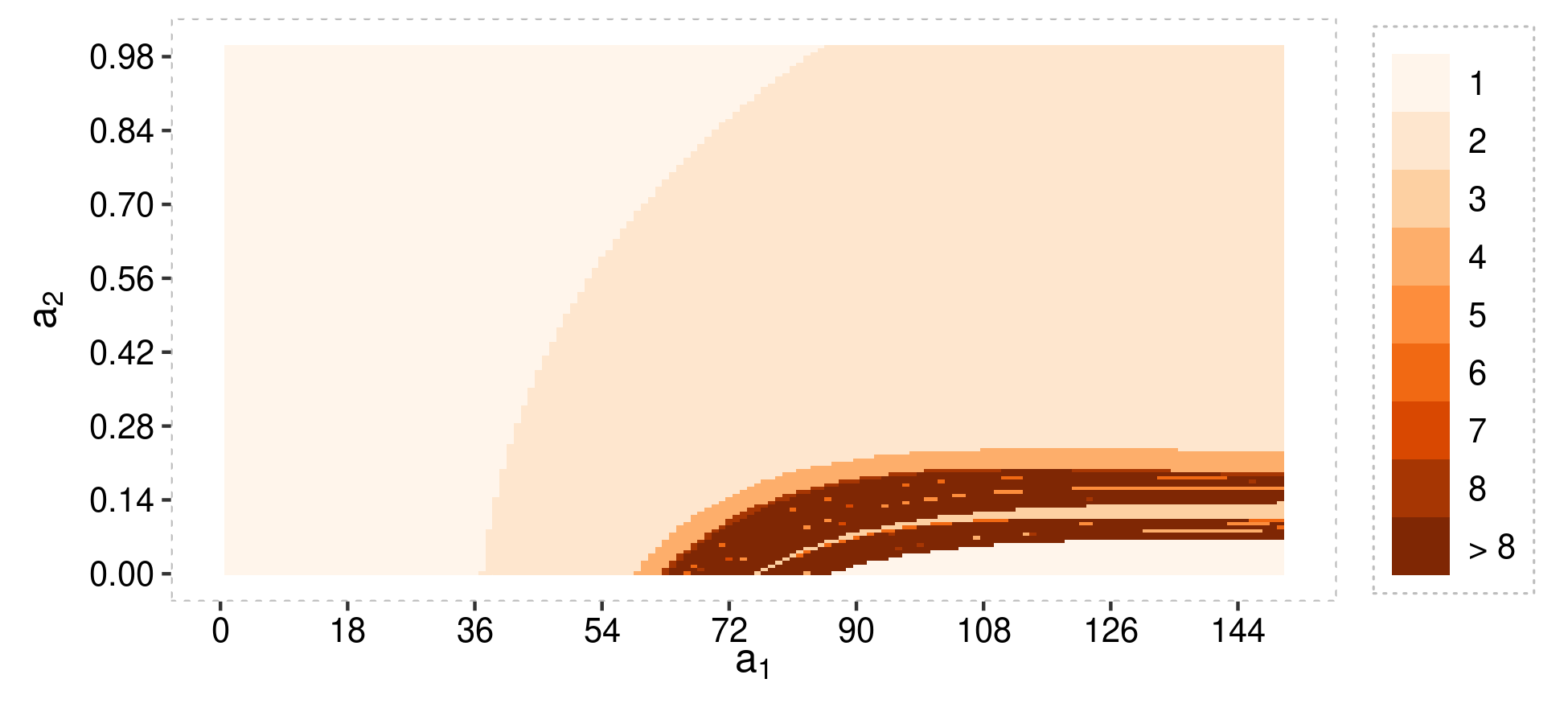}
\caption{Bifurcation diagram for the $(a_1,a_2)$-plane where regions 1 and 2 were respectively given by Ricker and Hassel-1 maps, for $a_1 \in (1,300]$, $a_2 \in (0,1)$. Initial conditions, $(x_1(0), x_2(0))=(131, 19)$. Other parameter values were as in Fig. \ref{xplots}. The colour-number legend shows the period of the various limit cycles ($2$-$8$), with $1$ representing convergence to a unique fixed point (either extinction or a positive equilibrium), and $>8$ representing periods higher than $8$ and possibly chaotic dynamics. Parameter grid resolution was $150 \times 150$.}
\label{aplotpdf}
\end{figure}

In order to produce a two-parameter bifurcation diagram for $a_1$ and $a_2$, we considered the last $100$ observations after $100,000$ iterations (see Fig. \ref{aplotpdf}). All other parameters were fixed as in Fig. \ref{xplots}. We produced bifurcation diagrams for various other initial conditions (different from the one in Fig. \ref{xplots}) and observed similar asymptotic behaviour. In Fig. \ref{aplotpdf} we let $a_1 \in (1,300]$ and $a_2 \in (0,1)$. The colour-number legend in Fig. \ref{aplotpdf} shows the period of the various limit cycles, with $1$ representing convergence to a unique fixed point (either extinction or a positive equilibrium), and $>8$ representing periods higher than $8$ and possibly chaotic dynamics, as chosen also in \cite{vortkamp}.

\begin{figure}[!h]
\centering
\includegraphics[trim = 10 15 10 15, clip, width = 13cm]{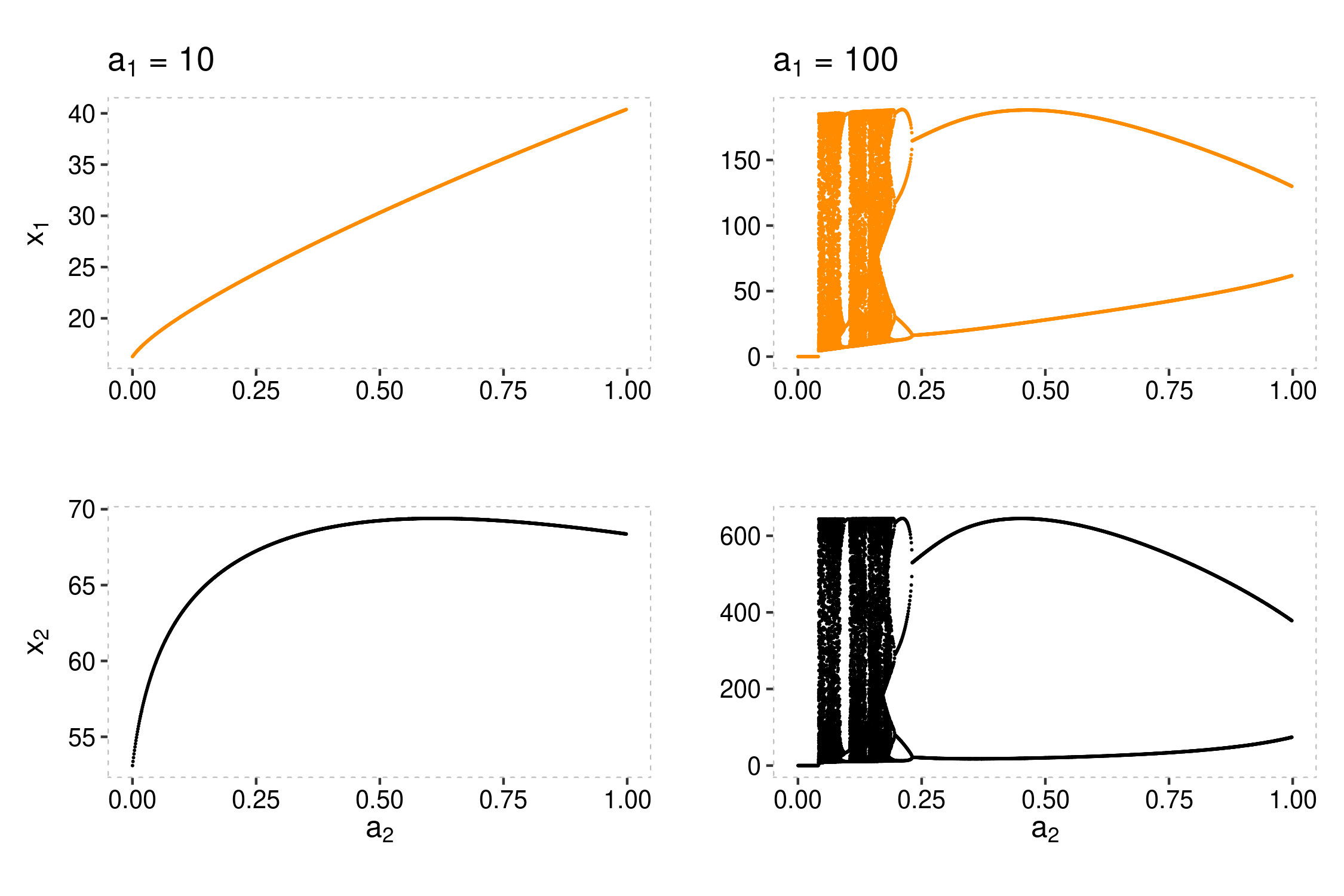}
\caption{ Bifurcation diagrams of $a_2 \in (0,1)$ versus $x_1$ (orange) and $x_2$ (black), where regions $1$ and $2$ were respectively Ricker and Hassell-1 maps. We considered the last $100$ observations after $10,000$ iterations when (left) $a_1=10$ and (right) $a_1=100$. Initial conditions where $(x_1(0), x_2(0))=(20, 10)$ for both scenarios.}
\label{bif_a1a2}
\end{figure}

In Fig. \ref{aplotpdf} we can observe that a large region of the $(a_1,a_2)-$parameter space consists of trajectories that either converge to a fixed point or enter period-$2$ limit cycles. For $a_2 \in (0, 0.28)$ we see a region where period-$2$ limit cycles emerge from chaotic type dynamics. For sufficiently low values of $a_2$ we see a portion of the $(a_1,a_2)$-plane where there is high sensitivity to changes in $a_1$. One moves from period-$2$ limit cycles to regions of higher period and potentially chaotic regimes, followed by regions where trajectories converge to extinction for low enough $a_2$. This diagram demonstrates the complexity that the model we consider can capture. That is, depending on the appropriate choice of both the dispersal and regional maps, deriving simple criteria for local/global stability of extinction or a positive equilibrium, in terms of $a_1$ and $a_2$, is not so trivial.

\vspace{3mm}

We will now observe how varying $a_i$, with $a_j$ fixed, for $i \neq j$, affects the asymptotic dynamics of (\ref{2by2model}). We will do this by plotting a single-parameter bifurcation plot for each $a_i$ vs $x_1$ and $x_2$. First let all parameters except $a_1$ and $a_2$, be as in Fig. \ref{xplots}. Similar to Fig. \ref{aplotpdf}, we considered the last $100$ observations after $100,000$ iterations. We first fix $a_1$ to a relatively low (Fig. \ref{bif_a1a2} (left)) and high value (Fig. \ref{bif_a1a2} (right)), while varying $a_2$ within $(0,1)$. We set $(x_1(0), x_2(0))=(20, 10)$. We produced single-parameter bifurcation diagrams for various other initial conditions and observed similar asymptotic behaviour. For $a_1=10$, Fig. \ref{bif_a1a2} (left), we observed (stable) fixed point regimes for all values of $a_2$ in the interval $(0,1)$. Once we increased $a_1$ to $100$ for small enough $a_2$ we observe stability of the extinction equilibrium and then the appearance of \textit{periodic windows}, where there was period halving bifurcations as $a_2$ increased toward $1$. For all values of $a_2 > c_2 \approx 0.24$, we see the appearance of period-$2$ limit cycles. Thus for a low enough range of $a_2$ values and high enough $a_1$ we see chaotic type dynamics. These then undergo a period-halving route to a (seemingly) locally stable period-$2$ limit cycle.

\begin{figure}[!h]
\centering
\includegraphics[trim = 10 15 10 15, clip, width = 13cm]{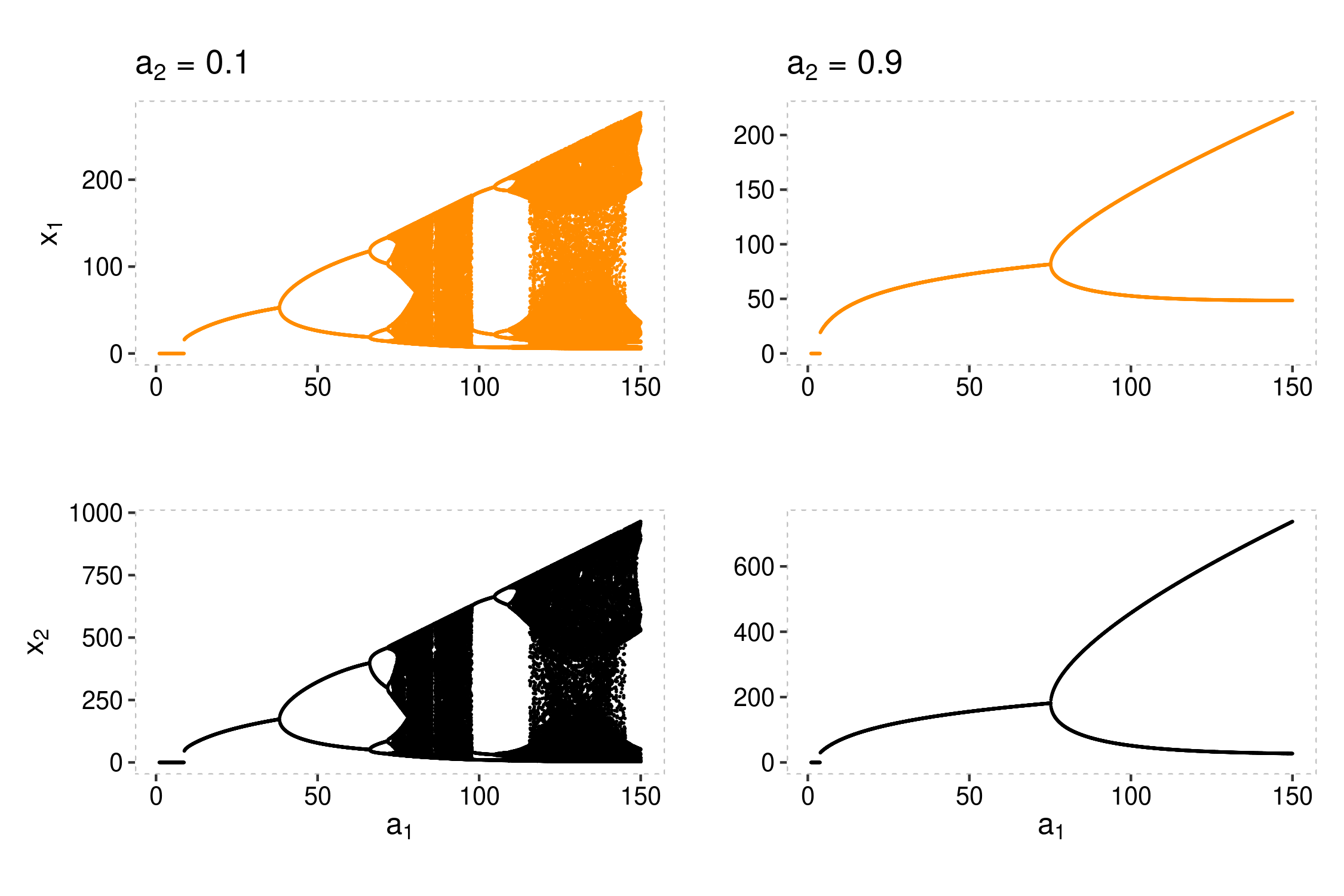}
\caption{ Bifurcation diagrams of $a_1 \in (1,150)$ versus $x_1$ (orange) and $x_2$ (black), where regions $1$ and $2$ were respectively Ricker and Hassell-1 maps. We considered the last $100$ observations after $10,000$ iterations when (left) $a_2=0.1$ and (right) $a_2=0.9$. Initial conditions where $(x_1(0), x_2(0))=(20, 10)$ in both scenarios.}
\label{bif_a2a1}
\end{figure}

We also fixed $a_2$ to a relatively low (Fig. \ref{bif_a2a1} (left)) and high value (Fig. \ref{bif_a2a1} (right)), while varying $a_1$ within $(1,150)$. We set $(x_1(0), x_2(0))=(20, 10)$, as in Fig. \ref{bif_a1a2}. In this case we observed significantly behaviour than in Fig. \ref{bif_a1a2}. For $a_2=0.1$ we observed that extinction was stable for sufficiently small $a_1$ values, with a stable positive fixed point emerging, followed by the appearance of periodic windows. As $a_1$ increased toward $150$ we observed period doubling bifurcations. We then enter a chaotic regime for large $a_1$ values where the dynamics become quite unpredictable. Once we increased $a_2$ to $0.9$ we saw that for low enough values of $a_1$ extinction is stable, and for sufficiently large values of $a_1$ less than $c_1 \approx 75$ we observe stability of a positive equilibrium followed by the emergence of a stable period-$2$ limit cycle.



\section{Conclusion}
Understanding the qualitative dynamics of sub-populations following coupling (via dispersal) is important for numerous ecological applications \cite{grombach, ecomod, kirkland}. In this paper, we considered a nonlinear model of dispersal, giving various sufficient conditions for asymptotic stability, persistence and existence of positive fixed points. We also briefly explored our results via numerical simulations and bifurcation diagrams. Many questions related to this model class are only partially understood. An obvious avenue for future research is to give sufficient conditions for (\ref{gendisp}) to have a unique, LAS or GAS positive equilibrium when each $f_i$ is in $\cal{M}$. Other possibilities also include characterising persistence and stability, and identifying what additional assumptions are needed so that one can guarantee dispersal-induced growth \cite{katriel}.


\section*{Acknowledgements}
BMC is supported by an IRC postgraduate scholarship (GOIPG/2020/939).

\end{document}